\documentclass{amsart}
\usepackage{amsmath}
\usepackage{amsthm,amssymb,amsfonts,hyperref}
\usepackage[all]{xy}

\title[IRS\MakeLowercase{s} in rank one]{On the growth of $L^2$-invariants of locally symmetric spaces, II: Exotic invariant random subgroups in rank one}

\author[Abert]{Miklos Abert$^1$}
\address{$^1$Renyi Institute of Mathematics
}
\email{abert.miklos@renyi.mta.hu}

\author[Bergeron]{Nicolas Bergeron$^2$}

\address{$^2$Institut de Math\'ematiques de Jussieu}
\email{bergeron@math.jussieu.fr}

\author[Biringer]{Ian Biringer$^3$}
\address {$^3$Boston College}
\email{ianbiringer@gmail.com  \rm (corresponding author)}

\author[Gelander]{Tsachik Gelander$^4$}
\address{$^4$Weizmann Institute}
\email{tsachik.gelander@gmail.com}

\author[Nikolov]{Nikolay Nikolov$^5$}
\address{$^5$University College, Oxford}
\email{zarkuon@gmail.com}

\author[Raimbault]{Jean Raimbault$^6$}
\address{$^6$Institut de Math\'ematiques de Toulouse}
\email{Jean.Raimbault@math.univ-toulouse.fr}

\author[Samet]{Iddo Samet$^7$}
\address{$^7$University of Illinois at Chicago}
\email{samet@math.uic.edu}

\theoremstyle{definition}

\newtheorem{theoalph}{Theorem}

\newtheorem{exalph}[theoalph]{Example}

\newtheorem{questalph}[theoalph]{Question}

\newtheorem{defn}[subsection]{Definition}

\swapnumbers
\newtheorem{thm}[subsection]{Theorem}
\newtheorem{lem}[subsection]{Lemma}
\newtheorem*{lem*}{Lemma}
\newtheorem{prop}[subsection]{Proposition}
\newtheorem*{prop*}{Proposition}

\newtheorem{example}[subsection]{Example}

\newtheorem{cor}[subsection]{Corollary}
\newtheorem*{theostar}{Theorem}
\newtheorem*{defstar}{Definition}

\newcommand {\comment} [1] {}

\newcommand{\length}{\mathrm{length}}
\newcommand{\diam}{\mathrm{diam}}

\newcommand{\BZ}{\mathbb{Z}}
\newcommand{\CT}{\mathbb{T}}
\newcommand{\CE}{\mathcal{E}}
\newcommand{\BG}{\mathcal{G}}
\newcommand{\EL}{\mathcal{EL}}

\newcommand{\PML}{\mathcal{PML}}
\newcommand{\BC}{\mathbb{C}}
\newcommand{\BR}{\mathbb{R}}
\newcommand{\modular}{\mathrm{Mod}}
\newcommand{\axis}{\mathrm{Axis}}
\newcommand{\BH}{\mathbb{H}}

\newcommand{\DD}{\mathcal{DD}}

\newcommand{\stab}{\mathrm{Stab}}

\newcommand{\Z}{\mathbb{Z}}
\newcommand{\N}{\mathbb{N}}

\DeclareFontFamily{OT1}{rsfs}{}

\DeclareFontShape{OT1}{rsfs}{n}{it}{<-> rsfs10}{}
\DeclareMathAlphabet{\mathscr}{OT1}{rsfs}{n}{it}

\newcommand{\Hom}{\mathrm{Hom}}

\newcommand{\R}{\mathbb{R}}
\newcommand{\SO}{\mathrm{SO}}
\newcommand{\Su}{\mathrm{SU}}

\newcommand{\rank}{\text{rank}}

\newcommand{\vol}{\mathrm{vol}}

\def\dist{\text{dist}}

\def\N{{\cal N}}

\newcommand\RR{{\mathcal{R}}}

\numberwithin{equation}{subsection}

\newcommand{\cal}{\mathcal}

\newcommand{\SL}{\mathrm{SL}}

\newcommand{\SU}{\mathrm{SU}}

\newcommand{\PSL}{\mathrm{PSL}}
\newcommand{\sub}{\mathrm{Sub}}

\begin{document}

\begin{abstract}
In the first paper of this series we studied the asymptotic behavior of Betti numbers, twisted torsion and other spectral invariants for sequences of lattices in Lie groups $G$.  A key element of our work was the study of \emph {invariant random subgroups} (IRSs) of $G$. Any sequence of lattices has a subsequence converging to an IRS, and when $G$ has higher rank, the Nevo--Stuck--Zimmer theorem  classifies all IRSs of $G$.   Using the classification, one can deduce asymptotic statments about  spectral invariants of lattices.

When $G$ has real rank one, the space of IRSs is more complicated. We construct here several uncountable families of IRSs in the groups $SO(n,1), \, n \ge 2$.  We give dimension-specific  constructions when $n=2,3$, and also describe a general gluing construction  that works for every $n$. Part of the latter contruction is inspired by Gromov and Piatetski-Shapiro's construction of non-arithmetic lattices in $SO(n,1)$.
\end{abstract} 
\maketitle


\section*{Introduction}

This is the second half of our paper \emph{On the growth of $L^2$-invariants of locally symmetric spaces}, which was posted on the arXiv in 2012  and which we have split in two for publication.  With the exception of this added introduction, all the sections of this paper appeared in the earlier preprint, and we have preserved the original section numbers so as to not break existing citations.

\subsection*{Invariant random subgroups}

Let $G$ be a locally compact second countable group, and $\sub_G$ the set of closed subgroups of $G$. We consider $\sub_{G}$ with the \rm Chabauty topology, \rm see \cite{Chabautylimite}.
%

\begin{defstar}
An \emph{invariant random subgroup} (IRS) of $G$  is a random element of $\sub_G$ whose law $\mu$ is a Borel probability measure on $\sub_G$ invariant under the conjugation action of $G \circlearrowright \sub_G$. Often, we will abusively  call $\mu$ an IRS as well.
\end{defstar}

The term IRS was introduced by Ab\'ert--Glasner--Vir\'ag in \cite{Abertkesten} for discrete groups, although they were also  studied by Vershik~\cite{Vershiktotally} under a different name, and we introduced IRSs to Lie groups in our earlier paper \cite{Abertgrowth}. If $G$ acts by a measure preserving transformations on a standard probability space $X$, then almost every stabilizer $G_x$, where $x\in X$, is a closed subgroup of $G $, see \cite[Theorem 3.2]{Varadarajangroups}. Hence, the stabilizer of a random $x\in X$ is an IRS of $G$. In fact, by \cite[Theorem 2.6]{Abert2012growth}, all IRSs can obtained as random stabilizers from p.m.p.\ actions. 

Concrete examples of IRSs include normal subgroups of $G$, as well as random conjugates $g\Gamma g^{-1}$ of a lattice $\Gamma < G$, where the conjugate is chosen by selecting $\Gamma g$ randomly against the given finite measure on $\Gamma \backslash G$. More generally, any IRS $H$ of a lattice $\Lambda < G$ \emph {induces} an IRS of $G$, obtained by conjugating $H$ by a random element of $\Gamma \backslash G$. We describe this in more detail in \S \ref{sec:induction} below. 

Though the definition of an IRS may seem rather unassuming, even in a very general context there are restrictions on which groups can actually appear in the support of one. For example an invariant random subgroup supported on amenable subgroups almost surely lies in the amenable radical \cite{Baderamenable}, and in the context of nonpositive curvature there is the stronger statement that an IRS is almost surely ``geometrically dense'' \cite{Duchesnegeometric}, which is a generalisation of the ``Borel density theorem'' proven in \cite{Abert2012growth}. There is a growing literature on IRSs on more specific groups, in particular construction of examples (see, e.g.,~\cite{Biringerunimodularity,bowen2012invariant,bowen2015invariant, thomas2014invariant, Gelander-LN}) and their applications, see especially \cite{bowen2010random, hartman2015furstenberg, tucker2015weak, WUD}. A still open question is whether there exists a simple, non-discrete locally compact group which does not have any IRSs beyond itself and the trivial subgroup; candidate are the ``Neretin groups'' which are known not to contain any lattice \cite{Badersimple}. 


\subsection*{Invariant random subgroups in semisimple Lie groups}

Suppose now that $G$ is a simple, noncompact center free Lie group. In \cite{Abertgrowth},  we proved the following  strong rigidity result  for higher rank IRSs using the Nevo--St\"uck--Zimmer Theorem \cite{Nevogeneralization,Stuckstabilizers} and Kazhdan's property (T).

\begin{theostar} \label{sz}
If $\rank_\BR G \geq 2$, the ergodic IRSs of $G$ are exactly $\{e\}, G$ and IRSs constructed as random conjugates of a lattice $\Gamma < G$. Moreover, if the laws of these IRSs are denoted $\mu_{id},\ \mu_G$, and $\mu_\Gamma$,  then for any sequence of pairwise non-conjugate lattices $\Gamma_n < G$, the measures $\mu_{\Gamma_n}$ weakly converge to $\mu_{id}$.
\end{theostar}

This completely describes the topological spaces of IRSs in simple higher-rank Lie groups. On the other hand, if $G$ is a group of real rank one then  there are always more IRSs than those described in this theorem. For example, every cocompact lattice in $G$ is a Gromov-hyperbolic group and hence contains plenty of normal infinite subgroups of infinite index, see Theorem 17.2.1 of \cite{Morrisintroduction}. Taking a random conjugate of one of the latter we obtain an IRS which violates the conclusion of the theorem above. More generally we can induce from IRSs in lattices, which yields further examples as many lattices in $\SO(n, 1)$ (and a few in $\SU(2,1)$) surject onto nonabelian free groups, which have plenty of IRSs by \cite{bowen2012invariant}. It turns out that in the case where $G = \SO(n,1)$ the wealth of available IRSs goes well beyond these examples, as hyperbolic geometry yields constructions of ergodic IRSs which are not induced from a lattice. These constructions are the main object of this paper and we will describe them in some detail in the next section. 

\medskip

First, though, let us say a few words about further restrictions on the groups appearing in IRSs of rank-one Lie groups. Making the Borel and geometric density theorems referenced above more precise, we will prove in Section \ref{sec:rank1}:

\begin{theoalph}[See Proposition \ref{limits} for a more general statement] Suppose that $\mu$ is an IRS  of $SO(n,1)$ that does not have an atom at $\{id\} \in \sub_G$.  Then $\mu$-a.e.\ $H\in \sub_G$  has  full limit set, i.e.\ $\Lambda(H)=\partial \BH^n$.
\end{theoalph}

In particular this forbids IRSs to be geometrically finite with positive probability. In the later work \cite{Abertunimodular} Ab\'ert and Biringer prove a much more precise result in dimension 3: any ergodic torsion-free IRS of $\SO(3,1)$ which is almost surely finitely generated must be either a lattice IRS or be supported on doubly degenerate subgroup (one of our constructions shows that the latter case contains a wealth of interesting examples). In dimension 2 any non-lattice ergodic torsion-free IRS must be supported on infinite-rank free subgroups; the topology of the corresponding infinite-type surfaces is very restricted, as described in \cite{Biringerends}. 

\medskip

Our main motivation in \cite{Abertgrowth} to study IRSs was their applications to the ``limit multiplicity problem''. As an example, a corollary of our classification of IRSs in higher rank is that for any sequence of cocompact lattices in an irreducible higher-rank Lie group $G$ with property $(T)$ which are in addition ``uniformly discrete'' (not intersecting a fixed identity neighbourhood in $G$---conjecturally this is always true if assuming torsion-freeness), the Betti numbers are always sublinear in the covolume except in the middle dimension if $G$ has discrete series.

This result is proven through the notion of \emph{Benjamini--Schramm convergence}, which gives a geometric meaning to the convergence of IRSs in the weak topology, and the generalization of the L\"uck Approximation Theorem to this setting. Since the space of IRSs is compact there are always limit points; thus, understanding the space of IRSs might in principle give information about the geometry and topology of locally symmetric spaces of large volume. 


\subsection*{Constructions of IRSs in real hyperbolic spaces}

To describe these examples, it will be more useful to interpret discrete\footnote{By a corollary of Borel's density theorem \cite[Theorem 2.9]{Abertgrowth}, any IRS of a simple Lie group $G$  that does not have an atom at $G\in \sub_G$ is discrete almost surely. 
} IRSs of $SO(n,1)$ as \emph{random framed hyperbolic $n$-orbifolds}.
   
Whenever  $\Gamma$  is a discrete subgroup of $SO(n,1)$,  the quotient $M_\Gamma =\Gamma \backslash \BH^n$ is a  hyperbolic $n$-orbifold. If a baseframe is  fixed in $\BH^n$,  its projection gives a canonical baseframe for $M_\Gamma$, and the map $\Gamma \mapsto  M_\Gamma$  is a bijection  from  the set of discrete subgroups of $G $ to the set of isometry classes of framed  hyperbolic $n$-orbifolds. A random $\Gamma$  then gives a random framed orbifold. 
 Intuitively, a random framed hyperbolic $n$-orbifold  represents an IRS if whenever a particular (unframed)  orbifold is chosen, the  base frames for that orbifold are distributed according to the natural Riemannian volume on the frame bundle. This is made precise in Abert--Biringer~\cite{Abertunimodular}, but here we will always just work with the IRS directly.

\begin {exalph}[IRSs from shift spaces, see \S \ref{GPS}]
Let $N_0,N_1$ be two hyperbolic $n$-manifolds with totally geodesic boundary, such that each boundary is the disjoint union of two isometric copies of some fixed hyperbolic $(n-1)$-manifold $\Sigma$. 
Given a
sequence $\alpha=(\alpha_i)_{i\in\mathbb{Z}}\in\{0,1\}^{\mathbb{Z}}$,
 construct a  hyperbolic $n$-manifold $N_{\alpha}$  by gluing copies of $N_0,N_1$
according to the pattern prescribed by $\alpha$.

If $\nu$  is a shift invariant measure on $\{0,1\}^\BZ$,  construct a random framed hyperbolic $n$-manifold by picking $\alpha$  randomly with respect to $\nu$,  and then picking a $\vol$-random base frame from the frame bundle of the `center' block, i.e.\  the copy of $N_0$ or $N_1$  corresponding to $0\in \BZ$.  This defines an IRS of $SO(n,1)$. Moreover, we show that if $N_0,N_1$ embedded in non-commensurable compact arithmetic $n$-manifolds $M_0,M_1$, and $\alpha $ is not supported on a periodic orbit, then the IRS is not induced by a lattice.\end {exalph}

 In some sense, these shift examples are still quite `finite', in that the manifolds involved are constructed as gluings of compact pieces.   So for instance, there is a universal upper bound for the injectivity radius  at every point of every manifold in the support of such an IRS. Here is an example that is more truly `infinite type', and  where the injectivity radius is not bounded above. Note:  one could also construct IRSs with unbounded injectivity radius by summing an  appropriate sequence of lattice IRSs against a geometric series, but  the following examples are ergodic.

\begin {exalph}[Random trees of pants, see \S \ref{trees}]
 Let $S $ be a topological surface constructed by gluing together pairs of pants glued together in the pattern dictated by an infinite $3$-regular tree. If  the set of simple closed curves on $S$  along which two pairs of pants are glued is written $\mathcal C$, then one can produce a hyperbolic metric on $S$  by specifying \emph {Fenchel--Nielsen coordinates}, which consist of a choice of length and twist parameter $(l_c,t_c) \in (0,\infty) \times S^1$  for each curve $c\in C$.
 
 Fix a probability measure on $(0,\infty)$, consider $S^1$  with the Lebesgue probability measure and fix a `center'  pair of pants on $S$. Construct a random framed hyperbolic surface by  choosing $(l_c,t_c)$  independently and randomly from $(0,\infty) \times S^1$, equipped with the product measure, and choosing a base frame $\vol$-randomly from the frame bundle of the center pair of pants.  The result is a random framed hyperbolic surface corresponding to an ergodic IRS of $SO(2,1)$. When $\nu$ is chosen appropriately, this IRS is not induced  from an IRS of a lattice, and when $\nu$  has unbounded support,  the injectivity radius at the base frame is not bounded above.
\end {exalph}

In three dimensions, we have already seen a nontrivial finitely generated IRS with infinite covolume: the IRS induced by the infinite cyclic cover $\hat M$ of a hyperbolic $3$-manifold fibering $M$ over the circle. This  cover $\hat M$  is homeomorphic to $S \times \BR$, where $S $ is the fiber, and geometrically it is \emph {doubly degenerate}, see \S \ref{manifolds}. In \S \ref{thick} we  generalize this example by showing that the basic idea behind the shift IRS construction can be performed in three dimensions without producing infinitely generated fundamental group. Namely,

\begin {exalph}[IRSs supported on surface groups, see \S \ref{thick}]
Let $\Sigma$  be a closed, orientable surface and $\phi_1,\ldots,\phi_n : \Sigma \longrightarrow \Sigma$ be  pseudo-Anosov homeomorphisms that freely generate a Schottky subgroup of $\mathrm{Mod}(\Sigma)$. Choose a sequence $(w_i)$ of words in the letters $\phi_1,\ldots,\phi_n $ with $|w_i|\to \infty$ and let $M_i$ be the hyperbolic mapping torus with monodromy $w_i$. 
 After passing to a subsequence, the measures $(\mu_{M_i})$  weakly converge  to an IRS $\mu$ of $SO(3,1)$ that is supported on uniformly thick, doubly degenerate hyperbolic $3$-manifolds homeomorphic to $\Sigma \times \BR$.  Moreover, if the words $(w_i)$ are chosen appropriately, no manifold in the support of $\mu$  covers a finite volume manifold, and in particular $\mu$ is not induced from an IRS of any lattice. 
\end {exalph}

The homeomorphisms $\phi_1,\ldots,\phi_n$ play the role of the building blocks $N_0,N_1$  in the shift examples, in the sense that  the  geometry of the mapping torus of each $\phi_k$ roughly appears as a `block' in the manifold $M_n$ whenever $\phi_k$  is a letter in $w_n$, and these blocks will persist in the limit IRS. Making this intuition precise requires us to use Minsky's model manifold machinery for understanding ends of hyperbolic $3$-manifolds up to bilipschitz equivalence \cite{Minskybounded,Minskyclassification1}, although in our setting we only need his earliest version for thick manifolds.  Although it seems plausible that one could use his machinery to construct the limit IRS directly, it is much easier to proceed as we do above through a limiting argument.


\subsection*{Further questions}

The topology of the space of IRSs of $SO(n,1)$ is quite rich---indeed, it is built on the already rich structure of the Chabauty topology. While a complete understanding of it (as in the higher rank case) is certainly intractable, we are interested in the following question, which we phrase more generally:
 
\begin {questalph}\label {sofic}Is every ergodic IRS of a simple Lie group $G $ other than $\mu_G $ a weak limit of IRSs corresponding to lattices of $G $?
\end {questalph}

It is easy to see that all the examples in $SO(n,1)$  that we have described are weak limits of lattice IRSs, and this is trivially true for higher rank $G$.  Question \ref{sofic}  is a continuous analogue of  a question of Aldous--Lyons \cite[Question 10.1]{Aldousprocesses},  which asks whether every unimodular random graph is a weak limit of finite graphs,  and which is in turn a generalization of the question of whether all groups are sofic.

While the doubly degerenate examples in dimension three are finitely generated, and there are no such examples in dimension 2,  the following question is open:

\begin{questalph} Are there non-lattice IRSs in $\SO(n, 1), n \ge 4$, which are finitely generated with positive probability? \end{questalph}

 Note that it is already unknown whether lattices in $SO(n,1)$ may have infinite, infinite index finitely generated normal subgroups.

For some groups there are no constructions of IRSs known beyond lattices and their normal subgroups. We may ask the following questions. 

\begin{itemize}
\item Are there non-lattice, almost surely irreducible ergodic IRSs in product groups such as $\PSL_2(\RR) \times \PSL_2(\RR)$?
\item Are there examples of ergodic IRSs in $\SU(n, 1), n \ge 1$, which are not induced from lattices?
\end{itemize}


\medskip

\noindent \textbf{Acknowledgments.} This research was supported by the MTA Renyi ``Lendulet" Groups and Graphs Research Group, the NSF, the Institut Universitaire de France, the ERC Consolidator Grant 648017, the ISF, the BSF and the EPSRC.

\setcounter{section}{10}


\section{Limit sets and induced IRSs}
\label{sec:rank1}


\subsection{Induction}\label {sec:induction}
To begin with, let $\Gamma$ be a
lattice in a Lie group $G$ and suppose that $\Gamma $ contains
a normal subgroup $\Lambda$; we construct an
IRS supported on the conjugacy class of $\Lambda$ as follows.  The map $G \ni g \mapsto g ^ {- 1} \Lambda g \in \sub_G $ factors as
$$G \longrightarrow \Gamma \backslash G \longrightarrow \sub_G, $$
and the second arrow pushes forward Haar measure on $\Gamma \backslash G$ to an IRS $\mu_\Lambda $ of $G $.  Geometrically, if
$D\subset G$ is a fundamental domain
for $\Gamma$ and $\overline{D}$ is its image in $G/\Lambda$, then a
$\mu_{\Lambda}$-random subgroup is the stabilizer in $G$ of a random point in
$\overline{D}$. If $\Lambda$ is a sublattice of $\Gamma $ this is coherent with the
previous definition.

This construction produces new examples of invariant random subgroups of rank one Lie groups.  Namely, lattices in $\R$-rank $1$ simple Lie groups are Gromov hyperbolic and it follows from \cite[Theorem 5.5.A]{GromovHyp} that they contain infinite, infinite index normal subgroups. In other words, the Margulis normal subgroup theorem fails for these
groups. Note that if $G =\SO(1,n)$ or
$\Su(1,n)$, this can be seen for example because in any dimension there are
compact real or complex hyperbolic manifolds with positive first Betti number. Now any such infinite, infinite index normal subgroup $\Lambda $ of a lattice in one of these groups $G $ gives an IRS supported on the conjugacy class of $\Lambda$.  In particular, Theorem \ref {sz} fails for these $G $.

The above is a special case of \emph{induction} of an IRS of  a lattice $\Gamma$ to an IRSs of $G $.    To define this more carefully, note that if  $\Gamma$ is a finitely generated group,  then an invariant random subgroup of $\Gamma $ is just a probability measure $\mu\in\mathcal{P}(\sub_{\Gamma})$ that is invariant under conjugation.  Examples include the Dirac mass at a normal subgroup, and the mean over the (finitely-many) conjugates of a finite index subgroup of a normal subgroup.  Less trivially, Bowen \cite {Bowen1} has shown that there is a wealth of invariant random subgroups of free groups.

So, let $\mu $ be an IRS of a lattice $\Gamma $ in a Lie group $G $. Define the IRS of $G $ \it induced \rm from $\mu$ to be the
random subgroup obtained by taking a random conjugate of $\Gamma$ and then
a $\mu$-random subgroup in this conjugate (which is well-defined because
of the invariance of $\mu$). Formally, the natural
map $$G\times\sub_{\Gamma} \ni (g,\Lambda) \longrightarrow g \Lambda g^{-1} \in\sub_G$$
factors through the quotient of $G \times \sub_\Gamma$ by the $\Gamma $-action
$(g,\Lambda)\gamma=(g\gamma, \gamma^{-1}\Lambda\gamma)$.  This quotient has a natural $G$-invariant probability measure, and we define our IRS to be the push forward of this measure by the factored map
$(G\times\sub_{\Gamma})/\Gamma\rightarrow\sub_G$.

\subsection {Limit sets of rank one IRSs}
We show in this section that IRSs in rank one groups have either full or empty limit set.  This is a trivial application of Poincar\'e recurrence that works whenever one has a reasonable definition of limit set.

Let $G$ be a simple Lie group with $\rank_\BR (G) = 1 $. The symmetric space $G/K $ is a Riemannian manifold with pinched negative curvature and therefore has a natural Gromov boundary $\partial_\infty X $.  The \it limit set \rm $\Lambda (H) $ of a subgroup $H < G $ is the set of accumulation points on $\partial_\infty X $ of some (any) orbit $H x $, where $x \in X $.  We say that $H $ is \it non-elementary \rm if $\Lambda (H)$ contains at least three points.

\begin {prop} \label {limits}
Suppose that $G$ is a simple Lie group with $\rank_\BR (G) = 1 $ and that $H < G $ is a closed, non-elementary subgroup.  Let $A \unlhd H $ be the compact, normal subgroup consisting of all elements that fix pointwise the union of all axes of hyperbolic elements of $H $.  Then if $\mu $ is an IRS of $H $, either
\begin {enumerate}
\item the limit set of a $\mu $-random $\Gamma < H $ is equal to $\Lambda (H)  $, or
\item $\mu \left (\sub_A \right) >0 $.
\end {enumerate}
\end {prop}

In particular, any IRS of $G $ without an atom at $id $ has limit set $\partial_\infty X $.  As a further example explaining condition (2), note that $H = \SO (2,1) \times \SO(3)$ embeds in $\SO (5, 1) $ and any IRS of $A=\SO (3) $ induces an IRS of $H $ with empty limit set.

\begin {proof}
Assume that with positive $\mu $-probability, the limit set of a subgroup $\Gamma < H $ is smaller than $\Lambda (H) $.  As $\Lambda (H) $ is second countable, there exists an open set $U \subset \Lambda (H) $ such that $\Lambda (\Gamma) \cap U = \emptyset $ with $\mu $-probability $\epsilon >0 $.  Since $H $ is non-elementary, there is a hyperbolic element $h \in H $ with repelling fixed point $\lambda_-\in U $ (see \cite [Theorem 1.1]{Hamenstadtrank}).  The element $h $ acts on $\partial_\infty X $ with north-south dynamics \cite [Lemma 4.4] {Hamenstadtrank}, and we let $\lambda_+ \in \partial_\infty X$ be its attracting fixed point.  Then for each $i $, the $\mu $-probability that $\Lambda (\Gamma) \cap h ^ i (U) = \emptyset $ is also $\epsilon $, by $H $-invariance of $\mu $.

Passing to a subsequence, we may assume that the sets $h ^ i (U) $ form a nested increasing chain with union $\partial_\infty X \setminus \lambda_+ $.  Therefore, passing to the limit we have that the $\mu $-probability that $\Lambda (\Gamma) \subset \{\lambda_+\} $ is $\epsilon $.  The $\mu $-probability that $\Lambda (\Gamma) = \{\lambda_+\} $ cannot be positive if $H $ is non-elementary, since the $H $-orbit of $\lambda_+ $ is infinite \cite [Theorem 1.1] {Hamenstadtrank} and there is equal probability of having limit set any translate of $\lambda_+ $.
Therefore, $\Lambda (\Gamma) = \emptyset $ with positive $\mu $-probability.

Suppose now that $(2) $ does not hold: then with positive $\mu $-probability we have $\Lambda (\Gamma) =\emptyset $ and $\Gamma \notin A $.  Pick some $\Gamma < H$ that satisfies these two conditions such that in every neighborhood of $\Gamma $ the $\mu $-probability of satisfying the two conditions is positive.  Since $\Gamma$ has empty limit set, $\Gamma$ must be finite and therefore has a nonempty fixed set $F_\Gamma \subset X $, which is a totally geodesic hyperplane in $X $.  Since $\Gamma $ is not contained in $A $, there is some hyperbolic element $h \in H $ whose axis is not contained in $F_\Gamma $.  Then fixing $x \in X $,
$$\max_{\gamma \in \Gamma} \,\dist \left(\gamma \circ h ^ i (x), h ^ i (x)\right) \  \longrightarrow \ \infty $$ as $i $ increases.
Moreover, this is true uniformly over some neighborhood $\mathcal U $ of $\Gamma \in \sub_H$.  Namely, for sufficiently small $\mathcal U $, we have that as $i $ increases,
$$\inf_{\text {finite }  \Gamma' \in \mathcal U} \ \ \max_{\gamma \in \Gamma'} \ \dist \left(\gamma \circ h ^ i (x), h ^ i (x)\right)\   \longrightarrow  \ \infty. $$
We can rephrase this by saying that as $i $ increases, $$\inf_{\text {finite }  \Gamma' \in \, h ^ {- i}\mathcal U h^ i} \ \ \max_{\gamma \in \Gamma'} \ \dist \left(\gamma (x), x\right)\   \longrightarrow  \ \infty. $$
Thus for an infinite collection of indices $i $, the subsets $h ^ {- i} \mathcal U h^ i $ are pairwise disjoint in $\sub_H $. This is a contradiction, since they all have the same positive $\mu $-measure.
\end{proof}

If $\Gamma < G $ is a subgroup whose limit set is not the full boundary $\partial_\infty G $, then there is no upper bound for the local injectivity radius $inj_{\Gamma \backslash X} (x) $ at points $x \in \Gamma\backslash X$ .  
Suppose that $\mu$ is an ergodic IRS of $G $. One can ask weather the function $$\mathrm {Sub}_{G} \to \BR, \ \ \Gamma \mapsto inj_{\Gamma \backslash X} ( [id]) $$ necessarily have finite $\mu$-expected value.  Here, $[id] $ is the projection of the identity element under $G \longrightarrow \Gamma \backslash X = \Gamma \backslash G / K $.
Dropping the ergodicity condition one can easily construct convex combinations of lattice IRSs with infinite expected injectivity radius.  Also, in Section \ref {trees} we construct ergodic IRSs that have \it unbounded \rm injectivity radius.


\section{Exotic IRSs in dimensions two and three}

\subsection {Random trees of pants --- examples in $G =\SO(2,1)$} \label {trees} The idea for this construction was suggested by Lewis Bowen.

Suppose that $S $ is a topological surface obtained by gluing together pairs of pants in the pattern dictated by a 3-valent graph $X $, i.e. properly embed $X $ in $\R ^ 3 $ and let $S $ be the boundary of some regular neighborhood of it.  Let $\mathcal C $ be the set of simple closed curves on $S $ corresponding to the boundary components of these pants.  Given a function $$\mathcal C \to (0,\infty) \times S ^ 1, \ \ c \mapsto (l_c, t_c),$$ called \it Fenchel-Nielsen coordinates, \rm we can construct a hyperbolic structure on $S$ by  gluing together hyperbolic pairs of pants whose boundary curves have lengths $l_c$, and where the two pairs of pants adjacent to a boundary curve are glued with twisting parameter $t_c$. The resulting structure is well-defined up to an isometry  that fixes the homotopy class of every curve in $\mathcal C$. See \cite {Matsuzakihyperbolic} for finite type surfaces, and \cite{Alessandrinifenchel}  for  a discussion of issues in the infinite type case.

Pick a Borel probability measure $\nu $ on $(0,\infty) $ and consider $S ^ 1 $ with Lebesgue probability measure $\lambda $.  We then have a probability measure $(\nu \times \lambda )^ {\mathcal C} $ on the moduli space $\mathcal M (S)$.  Fix some pair of pants $P\subset S $ bounded by curves in $\mathcal C $.   We create an IRS $\mu $ of $G$ as follows.  Randomly select an element $[d] \in \mathcal M (S)$, represented by a hyperbolic metric $d $ on $S $.  Let $P_{d} \subset S$ be the totally $d $-geodesic pair of pants in the homotopy class of $P $ and pick a base frame $f $ on $P_{d} $ randomly with respect to its Haar probability measure $m_{d} $.  The stabilizer $\stab(d,f) $ of $f $ under the $G $-action on the frame bundle of $(S, d) $ is a $\mu $-random subgroup of $G$.

More formally, if $A \subset \sub_G $ is a Borel subset, let
$$\mu (A) = \int_{d \in \mathcal M (S) } m_{d} \left ( \substack{ \text { frames } f \text { on } P_{d}\\ \text {with } \stab(d,f) \in A }\right ) \, d(\nu \times \lambda)^ {\mathcal C}. $$
\begin{prop}
If the $3 $-valent graph $X $ is vertex transitive, then $\mu $ is $G $-invariant.  Moreover, if $X $ is also infinite then $\mu $ is ergodic.
\end{prop}
\begin {proof}
Fix $g \in G $ and some Borel subset $A \subset \sub_G $.  Subdividing $A $  if necessary, we may assume that there is some pair of pants $P' \subset S $ bounded by curves in $\mathcal C $ such that if $d$ is a hyperbolic metric on $X $ as above and $f $ is a frame on $P_{d} $, then
$$\stab (d, f) \in A \implies gf \in P'_{d}, $$ where $P'_{d} $ is the totally $d $-geodesic pair of pants homotopic to $P' $.  Then we have
\begin {align*}
\mu (A) & = \int_{d \in \mathcal M (S)} m_{d} \left ( \substack{ \text { frames } f \text { on } P'_{d}\\ \text {with } \stab(d,f) \in gA }\right ) \, d(\nu \times \lambda)^ {\mathcal C} \\
& = \int_{d' \in \mathcal M (S)} m_{d'} \left ( \substack{ \text { frames } f \text { on } P_{d'}\\ \text {with } \stab(d',f) \in gA }\right ) \, d(\nu \times \lambda)^ {\mathcal C} \\
& = \mu (g A).
\end {align*}
Here, the first equality follows from our assumption on $A $ and the fact that $g $ acts as a measure preserving homeomorphism on the frame bundle of $(S, d) $.  For the seco nd, let $v \mapsto v' $ be a graph isomorphism of $X $ taking the vertex corresponding to $P $ to that corresponding to $P' $.  Then there is an induced map $d \mapsto d' $ on $\mathcal M (S)$; this map preserves the measure $(\nu \times \lambda)^ {\mathcal C}$, so the second inequality follows.

Now suppose that $X $ is infinite and $A \subset \sub_G $ is a $G $-invariant set.  Define $$\overline A = \left\{\, d \in \mathcal M (S) \, | \, \exists \text { a frame } f \text { on } P_{d} \text { with } \stab(d,f) \in A \, \right \}. $$
 The set $\overline A $ is invariant under the action on $\mathcal M (S)$ corresponding under Fenchel-Nielsen coordinates to the subgroup of $\mathcal C $-permutations arising from graph automorphisms of $X $.   As $X $ is infinite and vertex transitive, for every finite subset $F $ of $\mathcal C $ there is such a permutation such that $F $ is disjoint from its image.  It follows from a standard argument that the $(\nu \times \lambda)^ {\mathcal C}$-measure of $\overline A $ is either $0 $ or $1 $.  As $A $ is $G $-invariant, it can be recovered from $\overline A $ as the set of stabilizers of all frames on $(S, d) $ where $d $ ranges through $\overline A $.   Therefore, $\mu (A) $ is either $0 $ or $1 $.
\end{proof}

If $X $ is infinite and vertex transitive and $\nu$ is non-atomic and supported within $(0,\epsilon) $, where $\epsilon $ is less than the Margulis constant, the measure $\mu $ cannot be induced from a lattice.  For then with full $(\nu \times \lambda)^ {\mathcal C}$-probability, the length parameters of the Fenchel-Nielsen coordinates of a point in $\mathcal M (S)$ cannot be partitioned into finitely many rational commensurability classes.  In a finite volume hyperbolic surface, there are only finitely many closed geodesics with length less than $\epsilon $.  If a hyperbolic surface isometrically covers a finite volume hyperbolic surface, then the lengths of its closed geodesics that are shorter than $\epsilon $ can be partitioned into finitely many rational commensurability classes.  Therefore, at most a measure zero set of Fenchel-Nielsen coordinates give hyperbolic structures on $S $ that isometrically cover finite volume hyperbolic surfaces.  This shows that $\mu $ cannot be induced from a lattice.

There is one additional feature of this example that is of interest:

\begin{prop}\label {treeinjectivity}
If $X $ is an infinite $3 $-valent tree and $\nu $ has unbounded support, then the injectivity radius at the base frame of a framed hyperbolic surface has infinite $\mu $-essential supremum.
\end{prop}

To prove the proposition, we need the following lemma.

\begin {lem}
Suppose $l > 0 $ and $P $ is a hyperbolic pair of pants with geodesic boundary all of whose boundary components have length in $[l, l+1]$. Let $\gamma $ be a geodesic segment in $P $ that has endpoints on $\partial P $ but is not contained in $\partial P $.  Then $$\length (\gamma) \geq \sinh\left ( \frac 1 {\sinh (l) }\right).$$ Also, if the endpoints of $\gamma $ lie on the same component of $\partial P $ then $\length (\gamma)\geq \frac {l- 1}2 $.
\end {lem}
\begin {proof} It suffices to prove the lemma when $\gamma $ is a simple closed curve.  For the first part,
double $P $ to obtain a closed hyperbolic surface of genus $2 $.  There is a closed geodesic $\bar \gamma $ on this surface homotopic to the double of $\gamma $; this has length at most twice that of $\gamma $. The Collar Lemma \cite[Lemma 12.6]{Farbprimer} states that the radius $\sinh^ {- 1}\left (  1 /{\sinh (\frac 12 \length \bar \gamma)}\right) $-neighborhood of $\bar \gamma $ is an annulus; as some boundary curve of $P $ intersects $\bar \gamma $, this radius is at most $l $.  This gives the first inequality.

If the endpoints of $\gamma $ lie on the same component of $\partial P $, they partition that component into two arcs $\alpha $ and $\beta $.  Without loss of generality, $\length (\alpha) \leq \frac {l+ 1}2$.  But the concatenation $\gamma \cdot \alpha $ is homotopic to one of the other two boundary components of $P $, so it must have length at least $l $.  The lemma follows.
\end{proof}

\begin {proof}[Proof of Proposition \ref {treeinjectivity}]
 Recall that to pick a $\mu $-random framed hyperbolic surface, we choose length and twist parameters for each edge of $X $, produce from these a hyperbolic metric $d $ on the surface $S $ and then choose a base frame randomly from a totally geodesic pair of pants $P_d $ on $S $ corresponding to some root of $X $.  Fix some large $l ,R >0 $ such that $\nu ([l,l + 1] )>0 $.  Then with positive probability the length parameters for every edge in an $R $-ball around the root in $X $ are within $[l, l + 1] $.  It follows that the injectivity radius at any point $p \in P_d $ is at least $$\star=\min\left\{\, \frac {l - 1} 2,\ R\sinh\left ( \frac 1 {\sinh (l) }\right)\,\right\}.$$  To see this, note that the injectivity radius is realized as the length of a geodesic segment which starts and terminates at $p $.  Either this geodesic enters and leaves the same boundary component of some pair of pants with boundary lengths in $[l,l + 1] $, in which case the first estimate applies, or it passes through at least $R $ such pants and the second applies.  As $\star $ can be made arbitrarily large, the proposition follows.
\end{proof}

\subsection {IRSs of $SO(3,1) \cong \PSL_2\BC$ supported on thick surface groups}
\label{thick}

Suppose that $\Sigma $ is a closed, orientable surface of genus $g $.  In this section we construct a large family of IRSs of $\PSL_2\BC $ that are supported on subgroups $\Gamma $ with $\BH ^ 3/\Gamma $ homeomorphic to $\Sigma \times \BR $.  These examples are similar in spirit to those constructed --- in any dimensions --- by gluings in the following section, but have the added feature that they are supported on finitely generated subgroups of $\PSL_2\BC $.  

\vspace{2mm}

The authors would like to thank Yair Minsky for an invaluable conversation that led to this example.

\subsection{} The construction makes use of the action of the mapping class group $\modular (\Sigma) $ on the Teichmuller space $\CT (\Sigma) $; we refer to \cite {Matsuzakihyperbolic}, \cite {HubbardTeichmuller}, \cite {Farbprimer} for the general theory.  We identify $\CT (\Sigma) $ as the space of equivalence classes of hyperbolic metrics on $\Sigma$, where two metrics $d_0, d_1 $ are equivalent if there is an isometry $(\Sigma, d_0)\to (\Sigma, d_1)$ homotopic to the identity map.  We will sometimes denote elements of Teichmuller space as $(\Sigma,d) $ and sometimes as $X $, depending on context. The group $\modular (\Sigma)$ acts properly discontinuously on $\CT (\Sigma) $ and the quotient is the moduli space of all hyperbolic metrics on $\Sigma $.  Teichmuller space admits a natural \it Teichmuller metric \rm (see \cite {HubbardTeichmuller}), with respect to which $\modular (\Sigma) $ acts by isometries. Thurston has shown \cite {Thurstongeometry} how to give the union of $\CT (\Sigma)$ with the space of  projective measured lamination space $\PML (\Sigma) $ a natural topology so that the resulting space is homeomorphic to a ball of dimension $6g-6$, with $\CT (\Sigma) $ as the interior and $\PML (\Sigma) $ as the boundary.  This topology is natural, in the sense that the action of $\modular (\Sigma) $ on $\CT (\Sigma) $ extends continuously to the natural action of $\modular (\Sigma) $ on $ \PML (\Sigma) $.

Our construction of IRSs of $\PSL_2\BC $  relies on the following definition.

\begin{defn}[Farb-Mosher]\label {Schottky}
A finitely generated, free subgroup $F\subset \modular (\Sigma) $ is \emph{Schottky} if any orbit of the action of $F $ on $\CT (\Sigma) $ is \it quasi-convex, \rm  i.e.\ after fixing $X \in \CT (\Sigma) $, there is some $C >0 $ such that any Teichmuller geodesic segment that joins two points from the orbit $F (X)$ lies in a $C $-neighborhood of $F (X) $.
\end{defn}

\noindent
{\it Remark.} Farb and Mosher \cite {Farb} have shown that if $\phi_1,\ldots,\phi_n $ are pseudo-Anosov elements of $\modular (\Sigma) $ with pairwise distinct attracting and repelling laminations, then for all choices of sufficiently large exponents the elements $  \phi_1 ^ {e_1},\ldots,\phi_n ^ {e_n}$ freely generate a purely pseudo-Anosov Schottky subgroup of $\modular (\Sigma) $.

\medskip

Suppose from now on that $\phi_1,\ldots,\phi_n \in \modular (\Sigma)$ freely generate a Schottky subgroup $F \subset \modular (\Sigma) $.  Choose a sequence of finite strings
$$
e^ 1 = (e_1^ 1,\ldots, e_{n_1} ^ 1),
\ \ \ \ e ^ 2  = (e_1 ^ 2,\ldots, e_{n_2} ^ 2), \ \ \ \ \ldots
$$
with entries in $\{0,\ldots, n\} $ and let $C $ be the sub-shift of $\{0, \ldots, n\} ^\BZ $ consisting of strings all of whose finite substrings are contained in $e ^ i $ for some $i $. Set
$$f_i : \Sigma \to \Sigma ,\ \  \ f_i = \phi_{e_{n_i} ^ i } \circ \cdots \circ \phi_{e_1 ^ i}. $$  A celebrated theorem of Thurston \cite {Thurstonhyperbolic2} then implies that each \it mapping torus \rm  $$M_{f_i} = \Sigma \times [0, 1] / (x, 0) \sim (f_i (x), 1) $$ admits a (unique) hyperbolic metric.  We let $\mu_i $ be the corresponding IRS of $\PSL_2\BC $.

\begin {thm}\label {three-dimensionalIRS}
Any weak limit of a subsequence of $(\mu_i) $ is an IRS $\mu $ of $\PSL_2\BC $ that is supported on subgroups $\Gamma < \PSL_2\BC $ with $\BH ^ 3/\Gamma $ homeomorphic to $\Sigma \times \BR $.  Moreover, if the shift space $C $ does not contain periodic sequences, no subgroup $\Gamma < \PSL_2\BC $ in the support of $\mu $ is contained in a lattice of $\PSL_2\BC $.
\end {thm}

Before beginning the proof in earnest, we give a motivational outline.  The idea is to associate to use the Schottky group $\left < \phi_1,\ldots,\phi_n \right >$ to associate to every element $\gamma \in \{0, \ldots, n\} ^\BZ $ of the shift space a pair consisting of the following elements:
\begin {enumerate}
\item a hyperbolic $3$-manifold $N_\gamma $ homeomorphic to $\Sigma \times \BR $
\item a `coarse base point' $P_\gamma $, i.e. a subset of $N_\gamma $ with universally bounded diameter.
\end {enumerate}   Shifting a string $\gamma $ corresponds to shifting the base point of $N_\gamma $ and convergence of $\gamma_i \in\{0, \ldots, n\} ^\BZ $ corresponds to based Gromov Hausdorff convergence of the associated pairs $(N_{\gamma_i}, P_{\gamma_i}) $.  A periodic string with period $(e_1,\ldots, e_{k})$ corresponds to the infinite cyclic cover of the mapping torus $M_{\phi_{e_1}, \ldots,\phi_{e_k}}$, where the placement of the base point depends on the particular shift of the periodic string; moreover, no aperiodic string produces a hyperbolic $3$-manifold that covers a finite volume manifold.  If the IRSs $\mu_i $ limit to $\mu $ as in the statement of the theorem, then the support of $\mu $ consists of subgroups $\Gamma < \PSL_2\BC$ with $\BH ^ 3 / \Gamma $ a based Gromov Hausdorff limit of the mapping tori $M_{f_i} $ (in fact, of their infinite cyclic covers) under some choice of base points.    Using the correspondence above, we see that such $\BH ^ 3 / \Gamma $ arise from elements of $\{0,\ldots, n\} ^\BZ $ that are limits of periodic sequences used in producing the mapping tori $M_{f_i} $.  Varying the base points chosen on $M_{f_i} $ gives Gromov Hausdorff limits corresponding exactly to the elements of the sub-shift $C \subset \{0,\ldots, n\} ^\BZ $.  Therefore, as long as $C $ does not contain periodic sequences, no such Gromov Hausdorff limit can cover a finite volume hyperbolic $3$-manifold.

\vspace {2mm}

The correspondence between elements of $\{0,\ldots, n\} ^\BZ $ and coarsely based hyperbolic $3$-manifolds occupies most of the exposition to follow.  It will be convenient to embed the shift space inside of an auxiliary space, consisting of geodesics in the Schottky group $F=\left <\phi_1,\ldots,\phi_n\right> $.  Namely, we consider $F $ with its \it word metric: \rm
$$\dist (g, h) = \min \{ k \ | \ h ^ {- 1}g =\phi_{i_1} \ldots \phi_{i_k} \}.$$
and define the \it space of geodesics \rm in $F $ as the set
$$\BG(F) := \left \{ \gamma : \BZ \longrightarrow F \ | \ \gamma \text { word-isometric embedding} \right \}, $$
which we consider with the compact-open topology.  The space of geodesics $\BG (F) $ has a natural \it shift operator \rm defined by the formula $$S : \BG (F) \longrightarrow \BG (F), \ \ S (\gamma) \, (i) = \gamma (i - 1). $$ The group $F $ acts on $\BG (F)$ via $(g \cdot \gamma) (x) = g \gamma (x)$ and the quotient $\BG (F) / F$ can be identified with the space of geodesics $\gamma : \BZ \to F $ with $\gamma (0) = 1 $.   Note that the shift operator $S $ descends to another `shift operator', also called $S $, on $\BG (F)/F $.  Finally, there is then a natural shift invariant embedding $$ \{0, \ldots, n\} ^\BZ \longrightarrow \BG (F)/F, \ \ \ \ \  e \mapsto [\gamma_e]$$ determined by the constraint $\gamma_e (i) ^ {- 1} \gamma_e(i+1) = \phi_{e_i}$.

It now remains to relate elements of $\BG (F) / F $ to hyperbolic $3$-manifolds.  The key to this is Minsky's theorem \cite{Minsky} that $\epsilon $-thick doubly degenerate hyperbolic $3$-manifolds homeomorphic to $\Sigma \times \BR $ are modeled on geodesics in the Teichmuller space $\CT (\Sigma) $.  The plan for the rest of the section is as follows.  In section \ref {geodesics}, we show how to go from elements of $\BG (F) / F $ to Teichmuller geodesics.   Section \ref {manifolds} describes Minsky's theorem above, and the following section completes the relationship between elements of $\BG (F) / F $ and hyperbolic $3$-manifolds.  Section \ref {topology} shows that convergence in $\BG (F)/F $ translates to based Gromov Hausdorff convergence of hyperbolic $3$-manifolds.  We then indicate how shift-periodic elements of $\BG (F)/F $ correspond to infinite cyclic covers of mapping tori, and end with a section devoted to the proof of Theorem \ref {three-dimensionalIRS}.

\subsection{Geodesics in $F $ and thick Teichmuller geodesics}\label {geodesics}

Let $O: F \longrightarrow \CT (\Sigma) $ be a fixed $F $-equivariant map. It follows from \cite[Theorem 1.1]{Farb} that the map $O $ is a quasi-isometric embedding that extends continuously to an $F $-equivariant embedding $O:\partial_\infty F \longrightarrow \PML (\Sigma) $. The following lemma is implicit in \cite {Farb}.

\begin{lem}  \label{FM}
For sufficiently large $C >0 $, if $\gamma: \Z \to F $ is any geodesic then there is a Teichmuller geodesic $\alpha_\gamma : \BR \to \CT (\Sigma) $ such that the Hausdorff distance between $O \circ\gamma(\Z)$ and $\alpha_\gamma(\BR) $ is at most $C $.  If we require $$\ \ \ \lim_{t \to \infty} \alpha_\gamma (t) = \lim_{i \to \infty} O\circ\gamma (i), \ \ \ \ \lim_{t \to -\infty} \alpha_\gamma(t) = \lim_{i \to -\infty} O \circ\gamma (i), $$
then $\alpha_\gamma $ is unique up to orientation preserving reparameterization.
\end{lem}
\begin {proof}
By quasi-convexity of $O (F) \subset \CT (\Sigma)$, each of the geodesic segments $$\alpha_i =[O \circ \gamma (- i),O \circ\gamma (i)]\subset \CT (\Sigma) $$ is contained in the $\epsilon $-thick part
$$\CT_\epsilon (\Sigma) = \{ (\Sigma, d) \in \CT (\Sigma) \ | \ inj (\Sigma, d) \geq \epsilon \} $$
of $\CT (\Sigma) $ for some universal $\epsilon =\epsilon (F, O) $.  Therefore, Theorem 4.2 from \cite {Minskyquasi-projections} implies that there exists some $C= C (F, O) >0 $ such that $\{O \circ\gamma (-i)  ,\ldots,O \circ\gamma (i) \}$ is contained in a $C$-neighborhood of $\alpha_i $ for all $i $.  This means that all $\alpha_i $ pass within a bounded distance of $O \circ\gamma (0)  ,$ so Arzela-Ascoli's theorem guarantees that after passing to a subsequence $(\alpha_i) $ converges to a geodesic $\alpha_\gamma: \BR  \to \CT (\Sigma) $.  Then $O \circ\gamma (\Z)  $ lies in a $C $-neighborhood of $\alpha_\gamma(\BR), $ and thus also $\alpha_\gamma (\BR)$ lies in a $C' $-neighborhood of $O \circ\gamma (\Z)  $ for some $C' $ depending on $C$ and the distortion constants of the quasi-isometric embedding $O $.  The uniqueness is \cite [Lemma 2.4]{Farb}.
\end{proof}

Note in particular that $\alpha_\gamma$ has image contained in the $\epsilon $-thick part
$\CT_\epsilon (\Sigma)$ for some universal $\epsilon =\epsilon (F, O)$.

\subsection{} {\it Remark.}
Here one word of caution about Thurston's boundary of Teichmuller space is in order: if $\gamma : [0,\infty) \to \CT (\Sigma) $ is a geodesic ray, then it is not always true that $\gamma $ converges to a single point in the boundary $\PML (\Sigma) $ \cite {LenzhenTeichmuller}.  It is however true if $\gamma$ is contained in $\CT_\epsilon (\Sigma)$, the $\epsilon $-thick part of Teichmuller space, as
follows from three theorems of Masur \cite[Theorem 1.1]{MasurHausdorff}, \cite {Masurtwo}, \cite {Masuruniquely}. We state these as a lemma:

\begin {lem}[Masur]\label {convergence}
Fix $\epsilon >0 $ and let $\gamma: [0,\infty) \to \CT_{\epsilon} (\Sigma) $ be a geodesic ray. Then $\gamma $ converges to a point $\lambda \in\PML (\Sigma) $, and the lamination $\lambda $ is filling and uniquely ergodic.
Furthermore, if two geodesic rays $\gamma,\gamma': [0,\infty) \to \CT_\epsilon (\Sigma) $ converge to the same point of $\PML (\Sigma) $, then they are asymptotic.
\end{lem}

\subsection{Doubly degenerate hyperbolic structures on $\Sigma \times \BR $} \label {manifolds}  In this section we review some well-known facts about the geometry of hyperbolic structures on $\Sigma \times \BR $ and present Minsky's theorem that thick doubly degenerate hyperbolic structures are modeled on Teichmuller geodesics.  To begin, consider the deformation space $$AH (\Sigma) = \{ \rho : \pi_1 \Sigma \to \PSL_2\BC \ | \ \rho \text { discrete, faithful } \} / \PSL_2\BC, $$ where the quotient is with respect to the conjugation action of $\PSL_2\BC $.  We consider $AH (\Sigma) $ with the topology induced from the compact-open topology on the representation variety $\Hom (\pi_1\Sigma,\PSL_2\BC) $.  The space $AH (\Sigma) $ has a geometric interpretation as the set of `marked isometry types' of hyperbolic structures on $\Sigma \times \BR $:
$$AH(\Sigma) \cong \left \{ (N,\mu) \ \big | \ \substack {N \text { is a hyperbolic } 3-\text {manifold } \\ \text {homeomorphic to } \Sigma \times \BR}, \  \substack {\mu: \Sigma \longrightarrow N \text { is a} \\ \text {homotopy equivalence}}\right\}/ \sim,$$
where $(N_1,\mu_1) \sim (N_2,\mu_2) $ when there is an isometry $\phi: N_1 \to N_2 $ with $\phi \circ \mu_1 $ homotopic to $\mu_2 $.
Here, the correspondence assigns to a representation $\rho $ the quotient manifold $N = \BH ^ 3 / \rho (\pi_1\Sigma) $ and a marking $\mu : \Sigma \to N $ such that composing $\mu_* : \pi_1 \Sigma \to \pi_1 N $ with the holonomy map gives $\rho $ up to conjugacy.  We will use these two descriptions of $ AH (\Sigma) $ interchangeably.

Suppose now that $N $ is a complete hyperbolic $3$-manifold with no cusps that is homeomorphic to $\Sigma \times \BR $.  The two ends of $N $ admit a geometric classification: very loosely, an end of $N $ is \it convex cocompact \rm if level surfaces increase in area as they exit the end, while an end is \it degenerate \rm if the areas of level surfaces stay bounded.  We refer the reader to \cite {Matsuzakihyperbolic} for actual definitions and further exposition.

One says $N $ is \it doubly degenerate \rm if both its ends are degenerate.  We write
$$ \DD (\Sigma) = \{ \, (N,\mu) \in AH (\Sigma) \ | \ N \text { doubly degenerate} \} $$
for the space of all doubly degenerate elements of the deformation space $AH(\Sigma) $.

If $(N,\mu) $ is doubly degenerate, each of its ends has an \it ending lamination, \rm a geodesic lamination on $\Sigma $ that captures the geometric degeneration of level surfaces exiting that end (see \cite {Matsuzakihyperbolic}).  Ending laminations are always filling and they always admit a transverse measure of full support \cite [Lemmas 2.3 and 2.4]{Minskyrigidity}.  It is therefore convenient to defined the space $\EL (\Sigma) $ as the set of supports of filling measured laminations on $\Sigma $; in particular, the weak-* topology on (filling) measured laminations descends to a natural topology on $\EL (\Sigma) $.  We then have a function
$$\CE : \DD (\Sigma) \longrightarrow \EL (\Sigma) \times \EL (\Sigma)$$
that takes a doubly degenerate hyperbolic $3$-manifold to its two ending laminations.

The following theorem is well-known: the injectivity is Thurston's Ending Lamination Conjecture, recently resolved by Brock-Canary-Minsky \cite {Brockclassification2}, and the topological content follows from standard arguments and Brock's proof of the continuity of Thurston's length function \cite {Brockcontinuity}.  However, a nicely written proof of the full statement was recorded by Leininger and Schleimer in \cite {Leiningerconnectivity}.

\begin {thm}[Theorem 6.5, \cite {Leiningerconnectivity}] The map above gives a homeomorphism
$$ \CE : \DD (\Sigma) \longrightarrow \EL(\Sigma )\times\EL (\Sigma) \, - \, \Delta $$
onto the space of distinct pairs of elements of $\EL (\Sigma) $. \label {ELC}
\end {thm}

A fundamental result of Minsky \cite {Minskyrigidity} states that doubly degenerate hyperbolic $3$-manifolds with injectivity radius bounded away from zero are `modeled' on universal curves over bi-infinite geodesics in Teichmuller space.  This was an early part of Minsky's program to prove Thurston's Ending Lamination Conjecture, which as noted above was  established by Brock-Canary-Minsky in \cite {Brockclassification2}.  The following theorem is a major step in this work.

\begin{thm}[Minsky \cite {Minsky}] \label {models}
If $(N,\mu) \in \DD(\Sigma) $ is a doubly degenerate hyperbolic $3$-manifold and $inj (N) \geq \epsilon >0 $, there is some $C = C (\epsilon,\Sigma) $ and a bi-infinite Teichmuller geodesic $\alpha : \BR \longrightarrow \CT (\Sigma) $ with the following properties.
\begin {enumerate}
\item If $(\Sigma,d)$ is a point on the geodesic $\alpha(\BR) \subset \CT (\Sigma) $, there is a point $(\Sigma, d') \in \CT (\Sigma) $ at distance at most $C $ from $(\Sigma, d) $ and a pleated surface
$f_d: (\Sigma, d') \longrightarrow N $ in the homotopy class of $\mu $.
\item If $f: (\Sigma, d') \longrightarrow N $ is a pleated surface in the homotopy class of $\mu $, then the distance in $\CT (\Sigma) $ between $(\Sigma, d') $ and the image $\alpha(\BR) $ is at most $C $.

\end {enumerate}
Moreover, $\alpha $ limits to exactly two points on $\PML (\Sigma) = \partial \CT (\Sigma) $, the unique projective measured laminations supported on the two ending laminations $\CE (N,\mu)$.  If $$\CE (N,\mu) =\left (\lim_{t \to \infty} \alpha (t),\lim_{t\to-\infty} \alpha (t)\right), $$ then $\alpha $ is unique up to orientation preserving reparameterization.  In this case, we say that $(N,\mu) $ is \rm modeled on \it the geodesic $\alpha \subset \CT (\Sigma) $.
\end{thm}
\noindent Here, a \it pleated surface \rm $f: (\Sigma, d') \longrightarrow N $ is a map that is an isometric embedding on the complement of some geodesic lamination on $(\Sigma, d') $.  Note that in particular any pleated surface is $1 $-lipschitz.  We refer the reader to \cite {Matsuzakihyperbolic} for more details.

\medskip
\noindent
{\it Remark.} We should mention that the statement of Theorem \ref {models} given here is slightly different than that of \cite {Minsky}.  First of all, his theorem is more general since it deals with arbitrary hyperbolic $3$-manifolds rather than doubly degenerate structures on $\Sigma \times \BR $.  Also, his $(1) $ states that $(\Sigma, d) $ can be mapped into $N $ by a map with bounded \it energy, \rm rather than giving a pleated map from a nearby point in $\CT (\Sigma) $.  The version of $(1) $ above can be deduced from his $(1) $ and Proposition 6.2 from his paper.  Next, Minsky's statement does not reference uniqueness of $\alpha $ or its endpoints in $\PML (\Sigma) $.  However, if a geodesic $\alpha $ satisfies $(1) $ then it follows from Corollary 9.3 of \cite {Minskyrigidity} all accumulation points of $\alpha $ in $\PML (\Sigma) $ are supported on one of the ending laminations of $(N,\mu) $.  Furthermore, $(1) $ implies immediately that $\alpha $ is contained in some $\epsilon' $-thick part of Teichmuller space, so Theorem \ref {convergence} shows that the ending laminations of $(N,\mu) $ support unique projective measured laminations and that these are the endpoints of $\alpha $ in $\PML (\Sigma) $. Theorem \ref {convergence} shows that any two such $\alpha $ lie at bounded Hausdorff distance, so as any such $\alpha $ is contained in the thick part of Teichmuller space the uniqueness follows from \cite [Lemma 2.4] {Farb}.

\medskip

\begin {prop}\label {existence}
Suppose that $\alpha : \BR \longrightarrow \CT (\Sigma)$ is a geodesic in the $\epsilon $-thick part of $\CT (\Sigma) $. Then there is a unique  $(N,\mu) \in \DD (\Sigma) $ that is modeled on $\alpha $, in the sense of Theorem \ref {models}.  Moreover, $inj (N) \geq \epsilon' $ for some $\epsilon' > 0 $ depending only on $\epsilon. $
\end {prop}
\begin {proof}
Let $\lambda_+,\lambda_- \in \PML (\Sigma) $ be the endpoints of $\alpha $ given by Theorem \ref {convergence}.   Then by Theorem \ref {ELC}, there is a doubly degenerate hyperbolic $3$-manifold $(N,\mu) \in \DD (\Sigma) $ with ending laminations $\{|\lambda_+|,|\lambda_-|\} $.  Combining theorems of Minsky \cite {Minskybounded} and Rafi \cite {Raficharacterization}, we see that as $\alpha $ lies in the $\epsilon $-thick part of $\CT (\Sigma) $ there is a lower bound $inj (N) \geq \epsilon' >0 $ for the injectivity radius of $N $.  Then Theorem \ref {models} shows that there is some Teichmuller geodesic $\alpha' \subset \CT (\Sigma) $ on which $(N,\mu) $ is modeled.  But the uniqueness statement in Theorem \ref {convergence}  then implies that $\alpha' =\alpha $.
\end {proof}

\subsection{Geodesics in $F $ and hyperbolic $3$-manifolds} \label {BG} Now suppose that $\gamma \in \BG (F) $ and let $\alpha_\gamma: \BR \to \CT (\Sigma)$ represent the oriented Teichmuller geodesic at bounded Hausdorff distance from $O\circ\gamma (\BZ) $ given by Lemma \ref {FM}.  Let $$(N_\gamma,\mu) \in \DD (\Sigma)$$ be the unique doubly degenerate hyperbolic $3$-manifold modeled on the geodesic $\alpha_\gamma $, as given by Proposition \ref {existence}.  Then recalling that $\gamma (0) \in \modular (\Sigma) $, we define
$$\mu_{\gamma}= \mu \circ \gamma (0) : \Sigma \longrightarrow N_\gamma. $$
This gives a map $ \BG (F) \longrightarrow \DD (\Sigma) $ defined by $\gamma \mapsto (N_\gamma,\mu_{\gamma}) $.

\medskip
\noindent
{\it Remark.}
It follows from Proposition \ref {existence} that the manifolds $N_\gamma $ above all have injectivity radius $inj (N_\gamma) \geq \epsilon $ for some $\epsilon >0 $ depending only on $F $ and $O $.
\medskip

\subsection{} The map $: \BG (F) \longrightarrow \DD (\Sigma) $ factors as
$$\xymatrix{ \BG (F) \ar[rr] \ar [dr] & &  \DD (\Sigma) \\ & \BG (F)/F \ar[ur]& } $$
To see this, suppose that $\gamma \in \BG (F) $ and $g \in F $. Then $O (g \cdot \gamma) = g O (\gamma) $, so the nearby Teichmuller geodesics satisfy $\alpha_{g\gamma} = g \alpha_\gamma $.  However,  if $(N_\gamma,\mu) $ is modeled on $\alpha_\gamma $ then it follows immediately from the conditions in Theorem \ref {models} that $(N_\gamma,\mu \circ g ^ {- 1} )$ is modeled on $\alpha_{g\gamma} $.  So, $N_\gamma = N_{g\gamma} $ and $\mu_{g\gamma} = \mu \circ g ^ {- 1} \circ (g \cdot \gamma (0)) = \mu \circ \gamma (0) =\mu_\gamma. $
\begin {prop}\label {BGembeds}
The map  $ \BG (F)/F \to \DD (\Sigma) ,\  [\gamma] \mapsto (N_\gamma,\mu_\gamma) $ is an embedding.
\end {prop}
\begin {proof}
As mentioned above, the quotient map $\BG (F) \longrightarrow \BG (F) / F $ restricts to a homeomorphism on the subset $\BG_0 (F) $ consisting of geodesic $\gamma \in \BG (F) $ with $\gamma (0) = 1 $.  Using the notation of section \ref {BG}, for $\gamma \in \BG_0 (F) $ we have $\mu =\mu_\gamma $.  Therefore,
$$\xymatrix{ \BG_0 (F) \ar[rr]^ {\gamma \mapsto (N_\gamma,\mu_\gamma)} \ar [dr]_E & &  \DD (\Sigma) \\ & \EL (\Sigma) \times \EL (\Sigma) - \Delta\ar[ur]_{\CE^ {- 1}}& } $$
commutes, where $E $ takes a geodesic $\gamma : \BZ \to F $ to the pair of supports of $$\left (O \left(\lim_{t \to \infty}\gamma (t)\right), O \left(\lim_{t \to -\infty}\gamma(t)\right)\right) \in \PML (\Sigma) \times \PML (\Sigma) - \Delta$$ and $\CE $ is the map that takes a doubly degenerate hyperbolic $3$-manifold to its pair of ending laminations (see Theorem \ref {ELC}).  The maps $E $ and $\CE^ {- 1} $ are continuous and injective by Lemma \ref {FM} and Theorem \ref {ELC}, respectively.  As the domain $\BG_0 (F) $ is compact and the co-domain $\DD (\Sigma) $ is Hausdorff, our map is an embedding.
\end{proof}

\subsection{The topologies of $\BG (F)/F $ and $\DD (\Sigma)$} \label {topology}
Let $(\Sigma, d_0) = O (1) $ be the point in Teichmuller space that is the image of the identity $1\in F $.  We then have:

\begin {lem}\label {basepoints}
For all sufficiently large $C= C (F, O)>0 $, there is some $ D >0 $ as follows.  If $\gamma \in \BG (F) $ then the map $\mu_\gamma : \Sigma \to N_\gamma $ is homotopic to a $C $-lipschitz map $$m : (\Sigma, d_0) \longrightarrow N_\gamma. $$  Moreover, after fixing $\gamma $, the union $P_\gamma^ C =  \bigcup_m \, m(\Sigma) $ of the images of all such maps is a subset of $N_\gamma $ with diameter at most $D $.
\end {lem}
\begin {proof}
With the notation of section \ref {BG}, the point $O (\gamma (0)) $ lies at a bounded distance from the geodesic $\alpha_\gamma$.  By Theorem \ref {models} $(1) $, there is a pleated surface $f' : (\Sigma , d') \to N_\gamma $ in the homotopy class of $\mu $ such that $$\dist\big (\, (\Sigma, d'), \, O (\gamma (0))\, \big) \leq C', \\ \text { for some } C' = C' (F, O). $$ But then there is a pleated surface $f : (\Sigma, \gamma (0) ^ {- 1}_* d) \to N_\gamma $ given by $f = f' \circ \gamma (0) $ homotopic to $\mu_\gamma $.  Since the Teichmuller distance between $(\Sigma, \gamma (0) ^ {- 1}_*d) $ and $ (\Sigma, d_0) $ is bounded by $C' $, there is a $C $-lipschitz map $(\Sigma, d_0)\to (\Sigma, \gamma (0) ^ {- 1}_* d) $ homotopic to the identity map \cite{Choicomparison}.  Composing this with $f $ yields the map $m $ desired.

For the diameter bound on $P_\gamma $, fix an essential closed curve $a $ on $\Sigma $.  Then if $$m: (\Sigma,d_0) \longrightarrow N_\gamma $$ is a $C $-lipschitz map homotopic to $\mu_\gamma $, the image $m (a) $ is a closed curve in $N_\gamma $ with length at most $C\,  \length_{d_0} (a) $.  As the geodesic representative of $\mu_\gamma (a) $ has length at least $inj (N_\gamma) \geq \epsilon =\epsilon (F) >0 $, its distance to $m (a) $ is at most some $D' = D' (F) $.  Therefore,  $D = 2 D' + C \diam (\Sigma, d_0) $ is a bound for the diameter of the union $P_\gamma $.
\end{proof}

One way to interpret Lemma \ref {basepoints} is that one can regard the image of $ \BG (F) $ as a space of hyperbolic $3$-manifolds with preferred coarsely defined basepoints.  The basepoint for $(N_\gamma,\mu_\gamma) $ is just any point contained in $P_\gamma $; these basepoints are then well-defined up to a universally bounded error.
This viewpoint motivates the following lemma, which should be interpreted as saying that $\BG (F) $ maps continuously when the co-domain is given the topology of based Gromov-Hausdorff convergence.

\begin {lem}\label {basedcontinuous}
Suppose that $\gamma_i \to \gamma $ in $\BG (F) $ and let $C$ be as in Lemma \ref {basepoints}.  If $p_i \in P_{\gamma_i}^ C $ and $(N_{\gamma_i}, p_i) $ converges in the Gromov-Hausdorff topology to a based hyperbolic $3$-manifold $(N, p) $, then $N $ is isometric to $N_\gamma $.
\end {lem}
\begin {proof}
By Proposition \ref {BGembeds}, the pairs $(N_{\gamma_i},\mu_{\gamma_i}) $ converge to $(N_\gamma,\mu_\gamma) $ in $\DD (\Sigma) $.  Let $\rho_i : \pi_1 \Sigma \longrightarrow \PSL_2\BC $ be representations with $N_{\gamma_i} \cong \BH ^ 3 / \rho_i (\pi_1\Sigma) $ in such a way that the composition of $\mu_* : \pi_1 \Sigma \to \pi_1 N_{\gamma_i}$ with the holonomy map gives $\rho_i $, and assume by conjugation that $\rho_i $ converges to some $\rho_\gamma $ similarly associated with $(N_\gamma,\mu_\gamma) $.  As there is a universal lower bound for the injectivity radii $inj (N_{\gamma_i}) $ by Lemma \ref {FM} and Proposition \ref {existence}, the image $\rho_\gamma (\pi_1\Sigma) $ does not contain parabolic elements.  So, \cite[Theorem F]{Andersoncores}
implies that $\rho_i (\pi_1 \Sigma) $ converges to $\rho_\gamma (\pi_1\Sigma) $ in the Chabauty topology.

Fix some point $p \in \BH ^ 3 $ and denote its projection to quotients by $\bar p $.  By \cite[Theorem E.1.13]{Benedetti-Petronio} the manifolds $(N_{\gamma_i},\bar p) $ converge in the Gromov Hausdorff topology to a based hyperbolic $3$-manifold isometric to $N_\gamma $.  As noted in the proof of the lemma above, the points $p_i $ lie at bounded distance from the geodesic representatives in $N_{\gamma_i} $ of the loops $\mu_{\gamma_i} (a) $, where $a \subset \Sigma $ is any essential closed curve.  However, the same is true of the points $\bar p $ since these geodesic representatives are the projections of the axes of elements $\rho_i ([a]) \in \PSL_2\BC $, which converge as $i \to \infty $.  Therefore, the distances $\dist (\bar p, p_i) $ are universally bounded above, which implies immediately that the Gromov Hausdorff limit $N $ is isometric to $N_\gamma $.
\end{proof}

It is clear from the definitions that shifting a geodesic $\gamma \in \BG (F) $ does not change the isometry type of the manifold $N_\gamma $.  Adopting the viewpoint suggested by Lemma \ref {basepoints}, shifting $\gamma $ just changes the coarsely defined basepoint $P_\gamma $ for $N_\gamma $.  The following lemma shows that the basepoints $P_{S ^ n (\gamma)} $ of shifts of $\gamma $ coarsely cover $N_\gamma $.

\begin {lem}\label {density}
There exists some $C = C (F, O)>0 $ such that if $\gamma \in \BG (F) $ and $p \in N_\gamma $, then there exists some shift $S ^ n (\gamma) $ and an isometry $$i : N_\gamma \longrightarrow N_{S ^ n(\gamma)}$$ such that the distance between $i (p) $ and $P^ C_{S ^ n(\gamma)} $ is at most $C $.
\end {lem}
\begin {proof}
As in section \ref {BG}, assume that $(N_\gamma,\mu) $ is the element of $\DD (\Sigma) $ modeled on the Teichmuller geodesic at bounded Hausdorff distance from $\gamma (\BZ) $.  To prove the lemma, we need to show that for large $C $ and some $i \in \BZ $, there is a $C $-lipschitz map $(\Sigma, d_0) \longrightarrow N_\gamma $ in the homotopy class of $\mu \circ \gamma (i) $ whose image is within $C $ of $p $.

As $N_\gamma $ is doubly degenerate, it follows from work of Thurston \cite {Thurstongeometry} that through the point $p \in N_\gamma $ there is some essential closed curve with length at most some $C (\Sigma)  $.   Again by \cite {Thurstongeometry}, this closed curve is geodesically realized by a pleated surface $f : (\Sigma , d) \longrightarrow N_\gamma $ in the homotopy class of $\mu $.  Since $inj (N_\gamma) \geq \epsilon=\epsilon (F, O) $, the distance from $p $ to $f (\Sigma) $ is at most some $ C (C',\epsilon) =C (F, O) $.

By Theorem \ref {models} and Lemma \ref {FM}, there is a point $\gamma (i) \in \CT (\Sigma) $ at Teichmuller distance at most some $C (F, O) $ from $(\Sigma, d) $.  Therefore, there is a $C(F, O) $-lipschitz map $g : O (\gamma (i) ) \to (\Sigma, d) $ homotopic to the identity map \cite{Choicomparison}.  Then
$$\xymatrix{ (\Sigma , d_0 ) = O(\gamma (0)) \ar[rr]^(.6){\gamma (i)} & &  O (\gamma (i)) \ar [rr]^ {g \sim id} &  &(\Sigma, d) \ar[rr] ^ {f \sim \mu} & &N_\gamma} $$
composes to a $C $-lipschitz map homotopic to $\mu \circ \gamma (i) $ as desired.
\end{proof}

We use this lemma to prove the following proposition, which loosely states that Gromov Hausdorff limits can always be realized as limits in $\BG (F)/F $ up to shifts.

\begin {prop}\label {GromovHausdorfflimits}
Suppose $\gamma_i $ is a sequence in $\BG (F) $, $p_i \in N_{\gamma_i} $ and that $(N_{\gamma_i}, p_i) $ converges in the Gromov Hausdorff topology to $(N, p) $.  Then after passing to a subsequence, there is a sequence $(n_i) $ in $\BZ $ and some $\gamma \in \BG (F) $ such that
 $$[S ^ {n_i} (\gamma_i)] \longrightarrow [\gamma] \in \BG (F)/F $$ and $N_\gamma $ is isometric to $N $.
\end {prop}
\begin {proof}
By Lemma \ref {density}, we may assume after shifting each $\gamma_i $ that the base points $p_i $ lie in $P ^C_{\gamma_i} $.  Without changing their images in $\BG (F)/F $, we may translate the geodesics $\gamma_i $ so that $\gamma_i (0) = 1 $ for all $i $.  Then after passing to a subsequence, $(\gamma_i) $ converges to some $\gamma \in \BG (F) $, and the proposition follows from Lemma \ref {basedcontinuous}.
\end{proof}

\subsection {Shift periodicity and cyclic covers of mapping tori}  If $g: \Sigma \to \Sigma $ is a pseudo-Anosov homeomorphism, we define $\widehat M_g $ to be the infinite cyclic cover of
$$M_g = \Sigma \times [0, 1] / (x, 0) \sim (g (x), 1) $$ corresponding to the fundamental group of any of the level surfaces $\Sigma \times \{t\} $.  Note that the map $g $ is necessary to determine $\widehat M_g $ --- it is not sufficient to know only the homeomorphism class of $M_g $.  The main result here is the following:

\begin {prop}\label {axescover}\label {onlycover}
Suppose that $\gamma \in \BG (F) $.  If $[\gamma] \in \BG (F)/F $ is shift-periodic, then $N_\gamma $ is isometric to $\widehat M_g$ for some $g : \Sigma \to \Sigma $.  On the other hand, if $[\gamma] $ is not shift-periodic then $N_\gamma $ does not cover a finite volume hyperbolic $3 $-orbifold.
\end {prop}
\begin {proof}
Since $\gamma $ is shift-periodic in $\BG (F)/F $, one can easily check that it is the \it axis \rm of some element $g \in F $, i.e. $\gamma(\BZ) $ is invariant under the action of $g $ on $F $ by left translation and the restriction of $g $ to $\gamma(\BZ) $ is a nontrivial translation.

 Let $(N_\gamma,\mu) $ be the doubly degenerate hyperbolic $3$-manifold modeled on the Teichmuller geodesic at bounded Hausdorff distance from $O\circ \gamma (\BZ) $, as in Definition \ref {BG}.  Then $\mu_\gamma = \mu \circ \gamma (0) $ and $\mu_{g \gamma} = \mu \circ g \gamma (0) $.  Since $(N_\gamma, \mu_\gamma) $ and $(N_g\gamma,\mu_{g\gamma}) $ represent the same point of $\DD (\Sigma) $, this means that there is an isometry $i :N_\gamma \to N_\gamma$ with $i \circ \mu = \mu \circ g $.  As $g $ has infinite order, the isometry $i $ cannot have fixed points. Therefore, the quotient is a hyperbolic $3$-manifold $M $.  But $\pi_1 M $ is isomorphic to $\pi_1 M_g $, so a theorem of Waldhausen \cite {Waldhausenirreducible} implies that they are homeomorphic, in which case $N_\gamma $ is isometric to $\widehat M_g $.  This finishes the first part of the proposition.

For the second statement, suppose that $N_{\gamma} = N$ covers a finite volume hyperbolic $3 $-orbifold. Thurston's Covering Theorem\footnote{As Thurston's proof is not readily available, we refer the reader to \cite {Canarycovering} for a proof by Canary of a more general result.  Note that although Canary's statement does not deal with orbifold covers, the proof works just as well.} implies that $N_\gamma $ is isometric to the fiber subgroup of a mapping torus of $\Sigma $.  Therefore, there is some isometry $i : N_\gamma \longrightarrow N_\gamma $ with $i \circ \mu = \mu \circ f $ for some pseudo-Anosov homeomorphism $f: \Sigma \longrightarrow \Sigma $.  Here, $\mu: \Sigma \longrightarrow N_\gamma $ is the map given in section \ref {BG}.  Then it follows from Theorem \ref {models} that the action of $f $ on Teichmuller space leaves the geodesic $\alpha \subset \CT (\Sigma) $ on which $(N_\gamma ,\mu)$ is modeled invariant.

Fix some point $X $ on $\alpha \subset \CT (\Sigma) $.  As $O\circ \gamma (\BZ) $ lies at bounded Hausdorff distance from $\alpha $, for each $i \in \BZ $ there is some $j_i \in \BZ $ with
$$\sup_i \  \dist\left ( \, f ^ i (X), \, O\circ \gamma (j_i) \, \right) <\infty.$$
Therefore, by the equivariance of $O $ we have that
$$\sup_i \  \dist\left ( \, \gamma (j_i) ^ {- 1} f ^ i (X), \, O\circ \gamma (0) \, \right) <\infty.$$
Since the action of $\modular (\Sigma) $ on $ \CT (\Sigma) $ is properly discontinuous, this means that the set $\{\gamma (j_i )^ {- 1} f ^ i\, | \, i \in \BZ\} $ is finite.  In other words, we have some $i\neq k $ with $$g : = f ^ {k - i} = \gamma (j_i)\gamma (j_k) ^ {- 1} \in  F .$$  This means that there is some element $g \in F $ that acts as a nontrivial translation along the Teichmuller geodesic $\alpha $.   But recall from Lemma \ref {FM} that the extension $O: \partial _\infty F \longrightarrow \PML (\Sigma) $ is an embedding.  Therefore, as $O \circ \gamma (\BZ) $ and $O (\, \axis(g)\, ) $ accumulate to the same points of $\PML (\Sigma) $ we must have $\gamma (\BZ) =\axis(g)$.  Now if $\gamma (\BZ) $ is the axis of $g \in  F $, then for some $k $ we have $$g \gamma (i) = \gamma (i + k) = S ^ {-k} \gamma(i).$$  Then $\gamma $ and $S ^ {- k} (\gamma) $ have the same projection in $\BG (F)/F $, so $[\gamma] $ is periodic.
\end{proof}

\subsection{The proof of Theorem \ref {three-dimensionalIRS}}  For easy reading, we briefly recall the relevant notation.   We have pseudo-Anosov maps $\phi_0,\ldots,\phi_n \in \modular (\Sigma) $ that generate a Schottky subgroup $F < \modular (\Sigma) $.  The group $F $ acts on the space $\BG (F) $ of its geodesics and there is a shift-invariant embedding $$ \{0, \ldots, n\} ^\BZ \longrightarrow \BG (F)/F, \ \ \ \ \  e \mapsto [\gamma_e]$$ determined by the constraints
$\gamma_e (i) ^ {- 1} \gamma_e(i+1) = \phi_{e_i}$ and $\gamma_e (0) = 1 $.  Note that any $\gamma_e $ satisfying the first property has the same projection in $\BG (F)/F $, so the condition that $\gamma_e (0) = 1 $ is necessary only to uniquely specify $\gamma_e$ within its equivalence class.

As in the beginning of the section, choose a sequence of finite strings
$$
e^ 1 = (e_1^ 1,\ldots, e_{n_1} ^ 1),
\ \ \ \ e ^ 2  = (e_1 ^ 2,\ldots, e_{n_2} ^ 2), \ \ \ \ \ldots
$$
with entries in $\{0,\ldots, n\} $ and let $C $ be the sub-shift of $\{0, \ldots, n\} ^\BZ $ consisting of strings all of whose finite substrings are contained in $e ^ i $ for some $i $.

 \begin {lem}\label {obvious}Let $\bar e ^ i \in \{0,\ldots, n\} ^\BZ $ be a bi-infinite string obtained by concatenating copies of $e ^ i $.  Then the sub-shift $C $ consists of all the accumulation points of sequences $S ^ {n_j} (\bar e ^ {i_j})  $, where $n_j\in \BZ $ and $ i_j \in \N $ is increasing.  Consequently, if $$\bar e ^ i \mapsto [\gamma_i] \in \BG (F) /F\ \ \ \text { and }\ \ \ C \mapsto \Gamma_C \subset \BG (F)/F $$
then $\Gamma_C $ consists of all accumulation points in $\BG (F)/F $ of sequences $S ^ {n_j} (\gamma_{i_j})  $, where $n_j\in \BZ $ and $ i_j \in \N $ is increasing.
\end {lem}

Write $N_i $ for the manifolds $N_{\gamma_i} $ modeled on the geodesics in $F $ corresponding to the string $e^i$, where $\gamma_i $ is from the lemma above and $N_{\gamma_i} $ is defined in section \ref {BG}.
The mapping classes $f_i= \phi_{e_{n_i} ^ i } \circ \cdots \circ \phi_{e_1 ^ i}\in \modular (\Sigma)$ all pseudo-Anosov, so each \begin {align}\label {torus}M_i= \Sigma \times [0, 1] / (x, 0) \sim (f_i (x), 1) \end {align} has a hyperbolic metric, unique up to isometry.  By Proposition \ref {onlycover}, we identify $N_{i} $ with the infinite cyclic cover $\widehat M_g $ corresponding to $\pi_1 (\Sigma \times \{t\}) $.  We prove:

\begin {thm}\label{geometricversion}
Assume $p_i \in M_i $ and that some subsequence of $(M_i, p_i) $ converges in the Gromov Hausdorff topology to a hyperbolic $3$-manifold $(N, p) $.  If $N $ covers a finite volume hyperbolic $3 $-orbifold, then $C $ contains a shift-periodic point.
\end {thm}
\begin {proof}
Pick lifts $\hat p_i $ of $p_i $ in the cyclic covers $ N_i \longrightarrow M_i $.  We claim that a subsequence of $(N_i, \hat p_i) $ converges in the Gromov Hausdorff topology to $(N, p) $.  To see this, note that the projection onto the second factor in Equation \ref {torus} gives a map $M_i \to S ^ 1 $; define the \it circumference \rm of $M_i $  to be the length of the shortest loop that projects to a nontrivial element of $\pi_1 (S ^ 1) $.  In \cite  [Proposition 5.1] {Biringergeometry}, it is shown that there are only finitely many $\epsilon $-thick mapping tori with circumference less than a given constant.  The proof only uses that the covers $N_i $ are $\epsilon $-thick, though, which we know to be the case by Lemma \ref {FM} and Proposition \ref {existence}.  So, it follows that the circumferences of $M_i $ go to infinity.  In other words, there is an increasing sequence $r_i \in \BR $ such that the ball $B_{M_i }(p_i, r_i) $ is isometric to $B_{N_i} (\hat p_i, r_i) $ for all $i.$  It follows immediately that a subsequence of $(N_i, \hat p_i) $ converges to $(N, p) $.

Recall that $N_i = N_{\gamma_i} $, the manifold associated to $\gamma_i \in \BG (F) $.  By Proposition \ref {GromovHausdorfflimits}, there is a sequence $(n_i) $ in $\BZ $ and some geodesic  $\gamma \in \BG (F) $ with $$[S ^ {n_i} (\gamma_i)] \longrightarrow [\gamma] \in \BG (F)/F $$ and $N_\gamma $ isometric to $N $.  Since $N $ covers a finite volume hyperbolic $3 $-orbifold, $[\gamma ]$ is shift-periodic in $\BG (F)/F $ by Proposition \ref {onlycover}.  Lemma \ref {obvious} implies that $\gamma \in \Gamma_C $, but then as the map $\{0,\ldots, n\} ^\BZ \longrightarrow \BG (F) / F $ is a shift-invariant embedding, $\gamma $ must be the image of some shift-periodic point of $C $.  This finishes the proof.
\end{proof}

To conclude this section, we derive the statement of Theorem \ref {three-dimensionalIRS} given at the beginning of the section.  Suppose that $\mu_i $ is the IRS of $\PSL_2\BC $ corresponding to the hyperbolic $3$-manifold $M_i $.  If we write $M_i = \BH ^ 3/\Gamma_i $, then $\mu_i $ is supported on the set of conjugates of $\Gamma_i < \PSL_2\BC $.  Suppose that $\mu $ is the weak limit of some subsequence of $\mu_i $.  Then $\mu $ is supported within the set of accumulation points in $\mathrm {Sub}_{\PSL_2\BC} $ of sequences $g_i \Gamma_i g_i ^ {- 1} $, where $g_i \in \PSL_2\BC $.  But if $$g_i\Gamma_i g_i ^ {- 1} \longrightarrow \Gamma < \PSL_2\BC $$ then there are base points $p_i \in M_i $ and $p \in \BH ^ 3/\Gamma $ such that $(M_i, p_i) $ converges in the Gromov Hausdorff topology to $(\BH ^ 3/\Gamma, p) $.  Therefore, Theorem \ref {three-dimensionalIRS} follows from Theorem \ref{geometricversion}.


\section{A general gluing construction in $\SO(n,1)$}
\label{GPS}

The analysis of this section has been an inspiration for some later works regarding counting manifolds, which have already been published \cite{Raimbault13,CCC14}.

\subsection{Notation}

Let $N_0,N_1$ be two real hyperbolic $n$-manifolds such that each have
totally geodesic boundary, and each boundary is the disjoint union of two
copies of some fixed hyperbolic $(n-1)$-manifold $\Sigma$. Label
for each $N_a$ a component $\Sigma_a^+$ of $\partial N_a$, and
denote the other one by $\Sigma_a^-$; we call $i_{a}^{\pm}$ the corresponding
embeddings of $\Sigma$ in $\partial N_a$. Given a
sequence $\alpha=(\alpha_i)_{i\in\mathbb{Z}}\in\{0,1\}^{\mathbb{Z}}$ we let
$N_{\alpha}$ denote the manifold obtained by gluing copies of $N_0,N_1$
according to the pattern prescribed by $\alpha$:
$$N_{\alpha}=\left(\bigsqcup_{i\in\mathbb{Z}}N_{\alpha_i}\times\{i\}\right)/ (i_{\alpha_i}^+ x,i) \sim (i_{\alpha_{i+1}}^- x,i+1)\quad (i\in\mathbb{Z},x\in\Sigma).$$
For $i\in\mathbb{Z}$,  we shall denote by $N_{\alpha,i}$ the image of $N_{\alpha_k}\times\{i\}$ in
$N_{\alpha}$.
More generally, for an interval $I \subset \BZ$
set $N_{\alpha,I}= \cup_{i\in I} N_{\alpha,i}$.


\subsection{Construction of the IRS}

Let $\nu$ be a Borel probability measure on the Cantor set $\{0,1\}^{\mathbb{Z}}$. We define a measure $\mu_{\nu}$ on  the set of framed hyperbolic $n$-manifolds, and the IRS will be the corresponding measure on the set of discrete subgroups of $\SO(1,n)$  discussed in the introduction.

Let $\nu'$ be the measure on $\{0,1\}^{\mathbb{Z}}$, defined for Borel sets $A \subset \{0,1\}^{\mathbb{Z}}$ by
$$\nu'(A)=\frac{\int_A\vol(N_{\alpha_0})d\nu(\alpha)}{\int_{\{0,1\}^{\mathbb{Z}}}\vol(N_{\alpha_0})d\nu(\alpha)}.$$
 By definition, we obtain a $\mu_{\nu}$-random framed hyperbolic $n$-manifold by first choosing $\alpha $  randomly against $\nu'$, and then choosing a random base frame from $N_{\alpha,0}$.

\begin {example}
	Let $\sigma$ be the shift map on $\{0,1\}^{\mathbb{Z}}$ and suppose that $\nu$ is
a $\sigma$-invariant probability measure on $\{0,1\}^{\mathbb{Z}}$  that is supported on a periodic orbit, i.e.\ there is some $\alpha\in \{0,1\}^{\mathbb{Z}}$ and $k\in\mathbb{Z}$ with
$\sigma^k(\alpha)=\alpha$ and
$\nu=\frac 1 k \sum_{i=0}^{k-1}\delta_{\sigma^i(\alpha)}$. We can
construct a closed manifold $M$ from $\alpha$:
$$M=\left(\bigsqcup_{i\in\mathbb{Z}/k\mathbb{Z}}N_{\alpha_i}\times\{i\}\right)/ (i_{\alpha_i}^+ x,i) \sim (i_{\alpha_{i+1}}^- x,i+1) \quad (i\in\mathbb{Z}/k\mathbb{Z}).$$
Then each $N_{\beta}$, with $\beta$ in the support of $\nu$, is an infinite cyclic cover of $M$ and the random subgroup $\mu_{\nu}$ is the
ergodic IRS obtained --- as in \S \ref{sec:rank1} above --- from the normal subgroup $\pi_1(N_{\alpha})$ of the
lattice $\pi_1(M)$. 
\end {example}


 More generally,  we have the  following result.

\begin{lem}
Let $\nu$ be a shift-invariant ergodic measure on
$\{0,1\}^{\mathbb{Z}}$. Then the random subgroup $\mu_{\nu}$ constructed above
is an ergodic IRS.
\label{IRS}
\end{lem}

\begin{proof}
Fix $g\in G=\SO(1,n)$.   Translated through the correspondence between  framed manifolds and  discrete subgroups of $G$, the conjugation action of $g \circlearrowright \mathrm{Sub}_G$   restricts to the right action of $g$ on the  frame bundle $\mathcal FM = \Gamma \backslash G$ of any hyperbolic $n$-manifold $M=\Gamma \backslash \BH^n$.  Note that this action preserves the Haar measure.

 Let $\alpha \in \{0,1\}^\BZ$ and let $U$ be an open bounded set of
frames on $N_{\alpha}$. Then $U$ and $gU$ are contained in some
submanifold $N_{\alpha,I}$, where $I\subset \BZ$  is an interval.
Because  $\nu$ is shift-invariant, we get the same 
random  framed manifold  by selecting a random frame from $N_{\beta,I}$ for a
$\nu'$-random $\beta$. 

Let $V$ be some neighbourhood of $\alpha$ containing
all $\beta\in\{0,1\}^{\mathbb{Z}}$ such that for all $i=k,\ldots,k+l$ we have
$\beta_i=\alpha_i$; for $\beta\in V$ let $U_{\beta}$ be the image of
$U$ in $N_{\beta}$. Since $g$ preserves the Haar measure on $\mathcal{F}N_\alpha$,
when taking a random
frame in $N_{\beta,I}$ we have the same probability to
land in $U_\beta$ or $gU_\beta$, i.e.\ if we set
$$W=W(U,V)=\{y\in \mathcal{F}U_{\beta},\,\beta\in V\}$$
then we get
$$\mu_\nu(W)=\mu_\nu(\{y\in gU_\beta,\,\beta\in V\})=\mu_\nu(gW).$$
The $G$-invariance follows since the sets $\{y\in U_\beta,\,\beta\in V\}$
form a basis for Borel sets in the support of $\mu_\nu$.

 To show ergodicity, note that the group $G$ acts transitively on the frame bundle of any connected  hyperbolic $n$-manifold. So
if a $G$-invariant set $S$ of frames contains a frame on some
$N_\alpha$ it contains all frames on $N_\alpha$. It follows that  $\{\alpha \; : \; \Lambda_\alpha\in S\}$
is a shift-invariant set. Since $\nu$ is ergodic, it
follows that this set has full or zero measure.  Therefore $S$ has full or zero
measure for $\mu_\nu$.
\end{proof}
%

\subsubsection*{Remarks} 1) The IRSs we have constructed above are always limits of 
lattice IRSs since shift-invariant measures are limits of measures supported on 
finite orbits.

2) We could have made the construction with more
general graphs. If a group $\Omega$ acts freely on a locally finite
graph $T$, $D$ is a connected fundamental domain for $\Omega$ and
$N_0,N_1$ are manifolds with totally geodesic boundary whose boundary
components are all isometric and indexed by $\partial D$ then for
any $\alpha\in\{0,1\}^\Omega$ we can glue them along $T$ in the manner
prescribed by $\alpha$ to get a hyperbolic manifold. We can then
construct ergodic IRSs in the same manner as above from 
$\Omega$-ergodic probability measures on $\{0,1\}^\Omega$.


\subsection{Exoticity}
We now show that after choosing suitable $N_0$ and $N_1$, the construction above yields IRSs
that are not induced from a lattice.

\begin{thm}
Suppose that $n\ge 3$ and that $N_0$ (resp. $N_1$) is isometrically embedded 
in a compact arithmetic manifold $M_0$ (resp. $M_1$). If $M_0, M_1$ are 
noncommensurable then for any sequence $\alpha\in\{0,1\}^{\mathbb{Z}}$ that is 
not periodic the manifold $N_{\alpha}$
does not cover any finite volume hyperbolic manifold.
\label{Iddo}
\end{thm}

\begin{cor}
Under the same hypotheses as in the above theorem, 
if the ergodic shift-invariant measure $\nu$ is not supported on
a periodic orbit then the support of the IRS $\mu_\nu$ is disjoint from the
set of all subgroups of all lattices of $G$ (in particular it follows that
$\mu_\nu$ cannot be induced from a lattice).
\end{cor}

The proof of this theorem occupies the rest of this section. In \S \ref{standard},  we recall  how to construct (non-commensurable pairs of) arithmetic 
manifolds with totally geodesic hypersurfaces. These will be the manifolds $M_0,M_1$ above, and cutting along the hypersurfaces will give the desired $N_0,N_1$. The reason we use arithmetic manifolds here is the very strong   disjointness criterion in Proposition \ref {disjointness}, which says that isometric immersions of $N_0,N_1$ into a common manifold cannot have overlapping images.   Using this, we then show that if there is a covering map from $N_\alpha$ to a finite volume manifold, then $\alpha$  is periodic.

\subsection{Constructing arithmetic manifolds}
\label{standard}
The standard way to construct arithmetic hyperbolic manifolds that contain
totally geodesic hypersurfaces is as follows. Let $F$ be a totally real number
field and $q$ a quadratic form in $n+1$ variables over $F$ such that $q$ is
definite positive at all real places of $F$ but one, where it has signature
$(1,n)$. Then the group of integer points $\Gamma_q=\SO(q,\mathcal{O}_F)$ is a
lattice in $\SO(1,n)$. If $q$ is written as $a_1x_1^2+\ldots+a_{n+1}x_{n+1}^2$
where $a_1,\ldots,a_n$ are totally positive and $a_{n+1}$ is negative at
exactly one real place,
then $\Gamma_q$ contains the subgroup associated to the quadratic form in
$n$ variables $a_2x_2^2+\ldots+a_{n+1}x_{n+1}^2$ which gives rise to
an imbedded totally geodesic hypersurface. It follows from work of Millson
that there exists an ideal $\mathfrak{p}$ such that this
hypersurface is actually embedded in the manifold associated to the
principal congruence subgroup of level $\mathfrak{p}$, i.e.
$\Gamma\cap\ker(\SL(n+1,\mathcal{O}_F)\rightarrow\SL(n+1,\mathcal{O}_F/\mathfrak{p}))$.
Moreover we can choose $\mathfrak{p}$ so that this hypersurface $S$ is
non separating.  In this case, $M-S$ is  the interior of a compact manifold $N$  that has two boundary components, both isometric to $S$.
Note also that the isometry type of $S$ depends only on
$a_2,\ldots,a_{n+1}$ and $\mathfrak{p}$.

The simplest example of the previous procedure
is when $F=\mathbb{Q}$ and $a_1,\ldots,a_n>0,\, a_{n+1}<0$ but then the
manifolds obtained are noncompact for $n\ge 4$. However, if
$F=\mathbb{Q}(\sqrt{d})$ for a square-free rational integer $d>0$,
$a_1,\ldots,a_n\in\mathbb{Q}_+^*$ and $a_{n+1}/\sqrt{d}\in\mathbb{Q}^*$ then
$q$ is anisotropic over $F$ so that $\Gamma_q\backslash\mathbb{H}^3$ is
compact.

\label{fq}
Now we want to find $a_1,\ldots,a_{n+1}$ and $a_1'$ such that:
\begin{itemize}
\item Both $a_1,\ldots,a_{n+1}$ and $a_1',\ldots,a_{n+1}$ satisfy the
conditions above;
\item The lattices obtained from $q=a_1x_1^2+\ldots+a_{n+1}x_{n+1}^2$ and
$q'=a_1'x_1^2+\ldots+a_{n+1}x_{n+1}^2$ are noncommensurable.
\end{itemize}
By \cite[2.6]{GPS} it suffices to show that $q'$ and $\lambda q$ are not
isometric for any $\lambda\in F^*$. For $n$ odd, since the
discriminants of $q'$ and $\lambda q$ are equal for all $\lambda$ it suffices
that $a_1/a_1'\not\in F^2$ since then the discriminants of $\lambda q$ and
$q'$ are never the same (as noted in \cite{GPS}).
For example we can take $F=\mathbb{Q}(\sqrt{2})$ and
\begin{equation*}
q=x_1^2+\ldots+x_n^2-3\sqrt{2}x_{n+1}^2,\: q'=7x_1^2+\ldots+x_n^2-3\sqrt{2}x_{n+1}^2
\end{equation*}

For $n$ even we need to consider another invariant.
Let $k$ be any field; for $u,v\in k^*$ the Hilbert symbol $(u,v)$ is defined
in \cite[III,1.1]{CA} as $1$ if $1=uv^2+vy^2$ for some $x,y\in k$ and
$-1$ otherwise. Then it is shown in \cite[IV, Th\'eor\`eme 2]{CA} that
$$\varepsilon(q)=\prod_{i<j}(a_i,a_j)$$
is an isometry invariant of $q$ over $k$. Now we suppose that
$k=\mathbb{Q}_p$ for a prime $p>2$, then for $a,b\in\mathbb{Z}_p$ we have
$(a,b)=-1$ if and only if either $p$ divides $a$ (resp. $b$) and $b$
(resp. $a$) is
a nonsquare unit (modulo squares), or $a,b$ have the same $p$-valuation
mod $2$ and $-a^{-1}b$ is a square unit (see 
\cite[III, Th\'eor\`eme 2]{CA}). Now let $q$ and $q'$ be as above and
$\lambda\in\mathbb{Q}_7^*$. Since $7=3\pmod{4}$, $-1$ is not a square mod 7
and it follows that $(\lambda,\lambda)=1$, so that
$\varepsilon(\lambda q)=(\lambda,-\lambda\sqrt{2})^n=1$ since $n$ is even.
On the other hand, we have $-3\sqrt{2}=5\pmod{7}$, which is not
a square, so that $\varepsilon(q')=(7,-3\sqrt{2})=-1$. It follows that $q'$
and $\lambda q$ are not isometric over $\mathbb{Q}_7$ for any
$\lambda\in\mathbb{Q}_7$.

 In conclusion, this shows that we can find noncommensurable  compact arithmetic $n$-manifolds $M_0,M_1$ that both contain a totally geodesic hypersurface isometric to some fixed $S$, and then we can cut $M_0,M_1$ along to produce manifolds $N_0,N_1$  as required in the statement of Theorem \ref{Iddo}.


\subsection{The proof of Theorem \ref {Iddo}} Suppose that $n\geq 3$ and $N_0\subset M_0, N_1 \subset M_1$ are hyperbolic $n$-manifolds as in the statement of the theorem.   The reason we require $M_0,M_1$  to be non-commensurable arithmetic manifolds is the following.

\begin {prop}\label {disjointness}
	Suppose that $M $ is another hyperbolic $n$-manifold and $i_0: N_0 \longrightarrow M$ and $i_1: N_1 \longrightarrow M$ are isometric immersions.  Then the images of $i_0,i_1$ are disjoint, except  possibly along their boundaries.
\end {prop}

 To prove this, recall the following commensurability criterion
(see \cite[1.6]{GPS}).

\begin{lem}
If $\Gamma,\Gamma'$ are two arithmetic subgroups in $\SO(1,n)$ such that
the intersection $\Gamma\cap\Gamma'$ is Zariski-dense in $\SO(1,n)$, then
this intersection has finite index in both of them (so that  in particular $\Gamma,\Gamma'$ are commensurable).
\label{criterion}
\end{lem}

\begin {proof}[Proof of  Proposition \ref {disjointness}]
 Hoping for a contradiction, assume that the images of the interiors intersect.    To begin with, assume also that  there are components $\Sigma_0 \subset \partial N_0$ and $\Sigma_1 \subset\partial N_1$  such that 
$i_0(\Sigma_0) \cap i_1(int(N_1))\neq \emptyset$ and $i_1(\Sigma_1) \cap i_0(int(N_0))\neq \emptyset$.
Then the preimages $\Sigma_0'=i_1^{-1}(i_0(\Sigma_0))$ and $\Sigma_1' = i_0^{-1}(i_1(\Sigma_1))$ are  properly immersed totally geodesic hypersurface in $N_1, N_0$, respectively.   

 Fixing a monodromy map, identify $\pi_1(M)$  with a  discrete subgroup of $\SO(1,n)$. By choosing a base point within $N_0\cap N_1$, we can also select  subgroups in $\pi_1(M)$ that represent the fundamental groups of all the other manifolds and hypersurfaces above, such that
\emph {both $\pi_1 \Sigma_0' $ and $\pi_1 \Sigma_1' $ are contained in $\pi_1 N_1 \cap \pi_1 N_0$.} 

By Corollary 1.7.B of \cite{GPS},  for instance, the Zariski closure of $\pi_1 \Sigma_0'$ in $\SO(1,n)$  is isomorphic to $SO(1,n-1)$.   Similarly, the Zariski closure of $\pi_1 \Sigma_1'$  is a  (different) copy of $\SO(1,n-1)\subset \SO(1,n)$, so the group  $\langle \pi_1 \Sigma_0', \pi_1 \Sigma_1'\rangle$ generated by these two groups,   is Zariski dense in $\SO(1,n)$. 
 But $$\langle \pi_1 \Sigma_0', \pi_1 \Sigma_1'\rangle \subset \pi_1 N_0 \cap \pi_1 N_1,$$ so $\pi_1 M_0,\pi_1 M_1$  can be represented by lattices in $\SO(1,n)$ with Zariski dense  intersection. By Lemma \ref{criterion},  this contradicts that $M_0,M_1$  are not commensurable.

 The only remaining case is that, say, $i_0(\partial N_0)$  does not intersect $ i_1(int(N_1))$, so that $i_1(N_1) \subset i_0(int(N_0))$. In this case, though, $i_0^{-1}(i_1(N_1))$ is a  compact, immersed submanifold of $N_1$. So, some  finite index subgroup of $\pi_1 N_1 $  injects into $\pi_1 N_0$, and as $\pi_1 N_1 $  is Zariski dense in $\SO(1,n)$, we  get a contradiction just as before.
\end {proof}

 The appeal to Corollary 1.7.B of \cite{GPS} is the part of the argument above that uses $n\geq 3$.  If $n=2$, then $\Sigma_0'$   could be a geodesic segment,  in which case its (trivial) fundamental group is certainly not Zariski dense in $\SO(1,1)$.

 We are now ready to finish the proof of Theorem \ref {Iddo}, which we encapsulate in the following proposition.

\begin {prop}
Suppose that $\alpha =(\alpha_k) \in \{0,1\}^\BZ$ and that $f: N_\alpha \longrightarrow M $  is a covering map, where $M$  has finite volume.  Then $\alpha$  is periodic.
\end {prop}
\begin {proof}
 For convenience, assume throughout the following that $\alpha$  is not a constant sequence. Recall that $ N_{\alpha,i} \cong N_{\alpha_i}$ is the $i^{th}$ block in the gluing representing $N_{\alpha}$. Let $\Sigma_i$ be  the hypersurface that is the common boundary of $N_{\alpha,i}$ and  $N_{\alpha,i+1}$. All the $\Sigma_i$  are isometric to a fixed hyperbolic $(n-1)$-manifold $\Sigma$. 

By  Proposition \ref{disjointness},  we have that if $\alpha_i=0$ and $\alpha_j=1$,  then \begin {equation}
 	int(f(N_{\alpha,i}) \cap f(N_{\alpha,j})) = \emptyset. \label {disjointeq}
 \end {equation}
Hence $M=I_0 \cup I_1$,  where $I_0 = f(\cup_{i, \alpha_i=0} N_{\alpha,i}) $ and $I_1$  is defined similarly.   It follows that $I_0 \cap I_1$  is a set of totally geodesic hypersurfaces in $M$,  each of which is covered by $\Sigma$. A priori, you might imagine that the common boundary of $I_0 $ and  $ I_1$ has corners, remembering that the surfaces $f(\Sigma_i)$ are only immersed in $M$. However, if $\alpha_i=0$ and $\alpha_{i+1}=1$, say, then any transverse self intersection of the image of $f(\Sigma_i) $ would create   interior in the intersection $f (N_{\alpha_i}) \cap f(N_{\alpha_{i+1}} )$,  contradicting \eqref{disjointeq}.

Let's call a connected submanifold of $N_\alpha$  that is a   {maximal} union of   consecutive  blocks isometric to $N_0$ a \emph{0-chunk}, and define a \emph {1-chunk} similarly. The restriction of $f$ to any $0$-chunk  is a covering map onto some component $C\subset I_0$, and the degree of  this covering  is $ {2 \vol(\Sigma)}/{ \vol(\partial C)}.$ By  volume considerations, the number of blocks in any $0$-chunk that covers $C$  must then be
\begin{equation}
\frac{2 \vol(\Sigma)\cdot  \vol(C)}{\vol(N_0)\cdot \vol(\partial C) }.\label {numberblocks}	
\end{equation}
 Of course, all the  same  statements hold for $1$-chunks covering components of $I_1$.

From the covering property, every component $C\subset I_0$  has either one or two boundary components.  It follows that either
\begin {enumerate}
\item the (finitely many) components of $I_0$ and $I_1$  all have two boundary components, and are arranged in $M$ end-to-end in a circle, or
	\item  the components of $I_0$ and $I_1$  are arranged in a line segment, with one-boundary-component $C$'s at the extremities.
\end {enumerate}
When a  component $C $ has two boundary components, the two boundary components of a chunk covering $C$ cover distinct components of $\partial C$.  From this, it follows that $\alpha $ is periodic. Namely, a cyclic word representing $\alpha$  can be obtained  from the circle in case (1) by using \eqref{numberblocks} to determine the number of $0$'s and $1$'s to associate to each component of $I_0$ and $I_1$.   Case (2)  is similar, except that the cyclic word is produced by traversing the line segment twice,  first forward and then backward, but only counting the endpoints once each.
\end{proof}

\bibliographystyle{abbrv}
\bibliography{total2,affine,newbib}

\end{document}